\documentclass[12pt]{amsart}

\usepackage{color}
\usepackage{bbm}
\usepackage{graphicx}
\usepackage{amsfonts,amssymb,amsmath,amsthm,mathrsfs,enumerate,multicol,esint}

\usepackage[pagebackref,hypertexnames=false, colorlinks, citecolor=red, linkcolor=red]{hyperref} %,hypertexnames=false,colorlinks,[pagebackref]
\usepackage[backrefs]{amsrefs}

\begin{document}
\baselineskip 17.5pt
\hfuzz=6pt

\newtheorem{theorem}{Theorem}[section]
\newtheorem{prop}[theorem]{Proposition}
\newtheorem{lemma}[theorem]{Lemma}
\newtheorem{definition}[theorem]{Definition}
\newtheorem{cor}[theorem]{Corollary}
\newtheorem{example}[theorem]{Example}
\newtheorem{remark}[theorem]{Remark}
\newcommand{\ra}{\rightarrow}
\renewcommand{\theequation}
{\thesection.\arabic{equation}}
\newcommand{\ccc}{{\mathcal C}}
\newcommand{\one}{1\hspace{-4.5pt}1}

\def\RR{\mathbb R}

\newcommand{\al}{\alpha}
\newcommand{\be}{\beta}
\newcommand{\ga}{\gamma}
\newcommand{\Ga}{\Gamma}
\newcommand{\de}{\delta}
\newcommand{\ben}{\beta_n}
\newcommand{\De}{\Delta}
\newcommand{\ve}{\varepsilon}
\newcommand{\ze}{\zeta}
\newcommand{\Th}{\chi}
\newcommand{\ka}{\kappa}
\newcommand{\la}{\lambda}
\newcommand{\laj}{\lambda_j}
\newcommand{\lak}{\lambda_k}
\newcommand{\La}{\Lambda}
\newcommand{\si}{\sigma}
\newcommand{\Si}{\Sigma}
\newcommand{\vp}{\varphi}
\newcommand{\om}{\omega}
\newcommand{\Om}{\Omega}

%mathcalLIGRAPHIC
\newcommand{\cs}{\mathcal S}
\newcommand{\csrn}{\mathcal S (\rn)}
\newcommand{\cm}{\mathcal M}
\newcommand{\cb}{\mathcal B}
\newcommand{\ce}{\mathcal E}
\newcommand{\cd}{D}
\newcommand{\cdp}{D'}
\newcommand{\csp}{\mathcal S'}

%GENERAL
\newcommand{\lab}{\label}
\newcommand{\med}{\textup{med}}
\newcommand{\pv}{\textup{p.v.}\,}
\newcommand{\loc}{\textup{loc}}
\newcommand{\intl}{\int\limits}
\newcommand{\intlrn}{\int\limits_{\rn}}
\newcommand{\intrn}{\int_{\rn}}
\newcommand{\iintl}{\iint\limits}
\newcommand{\dint}{\displaystyle\int}
\newcommand{\diint}{\displaystyle\iint}
\newcommand{\dintl}{\displaystyle\intl}
\newcommand{\diintl}{\displaystyle\iintl}
\newcommand{\liml}{\lim\limits}
\newcommand{\suml}{\sum\limits}
\newcommand{\supl}{\sup\limits}
\newcommand{\f}{\frac}
\newcommand{\df}{\displaystyle\frac}
\newcommand{\p}{\partial}
\newcommand{\Ar}{\textup{Arg}}
\newcommand{\abssig}{\widehat{|\si_0|}}
\newcommand{\abssigk}{\widehat{|\si_k|}}
\newcommand{\tT}{\tilde{T}}
\newcommand{\tV}{\tilde{V}}
\newcommand{\wt}{\widetilde}
\newcommand{\wh}{\widehat}
\newcommand{\lp}{L^{p}}
\newcommand{\nf}{\infty}
\newcommand{\rrr}{\mathbf R}
\newcommand{\li}{L^\infty}
\newcommand{\rn}{\mathbf R^n}
\newcommand{\tf}{\tfrac}
\newcommand{\zzz}{\mathbf Z}

\newcommand{\MM}{\mathcal{M}}

%%% Notation for ( )
\newcommand{\contain}[1]{\left( #1 \right)}
\newcommand{\ContainA}[1]{( #1 )}
\newcommand{\ContainB}[1]{\bigl( #1 \bigr)}
\newcommand{\ContainC}[1]{\Bigl( #1 \Bigr)}
\newcommand{\ContainD}[1]{\biggl( #1 \biggr)}
\newcommand{\ContainE}[1]{\Biggl( #1 \Biggr)}

\allowdisplaybreaks

\title[$H^1_{\Delta_N}$ and $\rm BMO_{\Delta_N}$ via Weak Factorizations and Commutators]{Characterizations of $H^1_{\Delta_N}(\mathbb{R}^n)$ and ${\rm BMO}_{\Delta_N}(\mathbb{R}^n)$ via Weak Factorizations and Commutators}
\thanks{JL's research supported by ARC DP 160100153. BDW's research supported in part by National Science Foundation DMS \# 0955432 and \#1560955}
\thanks{{\it {\rm 2010} Mathematics Subject Classification:} Primary: 42B35, 42B25.}

\author{Ji Li}
\address{Ji Li, Department of Mathematics, Macquarie University, Sydney}
\email{ji.li@mq.edu.au}

\author{Brett D. Wick}
\address{Brett D. Wick, Department of Mathematics\\
         Washington University - St. Louis\\
         St. Louis, MO 63130-4899 USA
         }
\email{wick@math.wustl.edu}

\begin{abstract}
This paper provides a deeper study of the Hardy and $\rm BMO$ spaces associated to the Neumann Laplacian $\Delta_N$.  For the Hardy space $H^1_{\Delta_N}(\mathbb{R}^n)$ (which is a proper subspace of the classical Hardy space $H^1(\mathbb{R}^n)$) we demonstrate that the space has equivalent norms in terms of Riesz transforms, maximal functions, atomic decompositions, and weak factorizations.  While for the space ${\rm BMO}_{\Delta_N}(\mathbb{R}^n)$ (which contains the classical $\rm BMO(\mathbb{R}^n)$) we prove that it can be characterized in terms of the action of the Riesz transforms associated to the Neumann Laplacian on $L^\infty(\mathbb{R}^n)$ functions and in terms of the behavior of the commutator with the Riesz transforms.  The results obtained extend many of the fundamental results known for $H^1(\mathbb{R}^n)$ and $\rm BMO(\mathbb{R}^n)$.
\end{abstract}

\maketitle

 %\tableofcontents

\bigskip

\section{Introduction and Statement of Main Results}
\setcounter{equation}{0}

The spaces $H^1(\mathbb{R}^n)$ and $\rm BMO(\mathbb{R}^n)$ are fundamental function spaces in harmonic analysis.  The work of Fefferman and Stein, \cite{FS}, provides a duality relationship between $H^1(\mathbb{R}^n)$ and $\rm BMO(\mathbb{R}^n)$.  And, further provides characterizations of these spaces in terms of maximal functions, square functions, and Riesz transforms.  While the work of Coifman, Rochberg and Weiss, \cite{CRW}, provides a connection between weak factorization of the Hardy spaces, commutators with Riesz transforms and $\rm BMO(\mathbb{R}^n)$.  The main goals of this paper are to provide similar connections for $H^1$ and $\rm BMO$ spaces adapted to a particular linear differential operator.

There is a substantial literature related to $H^1$ and $\rm BMO$ spaces adapted to a linear operator $L$ on $L^2(\mathbb{R}^n)$ which generates an analytic semigroup $e^{-tL}$ on $L^2({\mathbb R}^n)$ with a  kernel $p_t(x,y)$ satisfying an upper bound.  That is, operators $L$ for which the kernel of the semigroup $p_t(x,y)$ there exists positive constants $m$ and $\epsilon$ such that for all $x,y\in {\mathbb R}^n$ and for all $t>0$:
\begin{equation}
\label{upper bound of pt}
|p_t(x,y)|\leq \frac{Ct^{\frac{\epsilon}{m}}}{(t^{\frac{1}{m}}
+|x-y|)^{n+\epsilon}}.
\end{equation}

In \cite{AuDM}, Auscher, Duong, and McIntosh defined a Hardy space $H^1_L({\mathbb R}^n)$ associated with such operators $L$ as the class of all functions $f\in L^1(\mathbb{R}^n)$ for which $S_L(f)\in L^1(\mathbb{R}^n)$, where $S_L(f)$ is Littlewood--Paley area function defined as follows.
\begin{equation}
\label{area function}
S_L(f)(x)= \bigg( \int_0^{\infty}\!\!\!\int_{|y-x|<t}  |Q_{t^m}f(y)|^2  \,\frac{dydt}{t^{n+1}} \bigg)^{\frac{1}{2}},
\end{equation}
with $Q_t=tLe^{-tL}$.  The $H^1_L({\mathbb R}^n)$ norm of $f$ is defined as $\|f\|_{H^1_L({\mathbb R}^n)}=\|S_L(f)\|_{L^1({\mathbb R}^n)}$. %They then obtained an $L$-molecular characterization  for $H^1_L({\mathbb R}^n)$  by using the theory of tent spaces developed in \cites{CMS1, CMS2}.

In \cites{DY1,DY2}, Duong and Yan defined the function space ${\rm BMO}_L({\mathbb R}^n)$ associated with an operator $L$.  They then go on to prove that if $L$ has a bounded holomorphic functional calculus on $L^2(\mathbb{R}^n)$  and the kernel $p_t(x,y)$ of the semigroup $e^{-tL}$ satisfies the upper bound (\ref{upper bound of pt}), then the space ${\rm BMO}_{L^{\ast}}({\mathbb R}^n)$ is the dual space of Hardy space $H^1_L({\mathbb R}^n)$ in which  $L^{\ast}$ denotes the adjoint operator of $L$. This gives a generalization of the duality of  $H^1({\mathbb R}^n)$ and ${\rm BMO}({\mathbb R}^n)$ of Fefferman and Stein \cite{FS}. Later, the theory of function spaces associated with operators has been developed and generalized to many other different settings, see for example \cite{HM,AMR,HLMMY,JY,DL}.

The choice of $L=\Delta$ gives rise to the spaces to the classical spaces $H^1(\mathbb{R}^n)$ and $\rm BMO(\mathbb{R}^n)$.  While the choice of the semigroup $e^{-tL}$ is the Poisson  semigroup $e^{-t\sqrt{\Delta}}$ (here $m=2$), given by
\begin{eqnarray}\label{classical case}
\ \ \ \ \  e^{-t\sqrt{\Delta}}f(x)=\int_{{\mathbb
R}^n}p_{t}(x-y)f(y)\,dy,\ t>0, \ \ {\rm where}\ \ p_{t}(x)=\frac{c_nt}{
(t^2+|x|^2)^{\frac{n+1}{2}} }
\end{eqnarray}
yields the spaces $H^1_{\sqrt{\Delta}}({\mathbb R})$
and ${\rm BMO}_{\sqrt{\Delta}}({\mathbb R})$ coincide with the
classical Hardy space and BMO space, respectively (see \cite{AuDM} and
\cite{DY1}).

In \cite{DDSY}, Deng, Duong, Sikora, and Yan further considered the
comparison of ${\rm BMO}_{L}({\mathbb R}^n)$ and ${\rm BMO}({\mathbb
R}^n)$. By considering the Neumann Laplacian $L=\Delta_N$, they obtained that
$$
{\rm BMO}({\mathbb R}^n) \subsetneq {\rm BMO}_{\Delta_N}({\mathbb R}^n).
$$
Recently, in \cite{Yan} Yan introduced a class of $H^p_L(\mathbb{R}^n)$ for a range of $p\leq1$ by using the Littlewood--Paley area function $S_L(f)$.  In particular, Yan showed that
$$
H^p_{\Delta_N}(\mathbb{R}^n)\subsetneq H^p(\mathbb{R}^n), \ \ \ \frac{n}{n+1}<p\leq 1.
$$

The main goal of this paper is to carry out a deeper study of the spaces $H^1_{\Delta_N}(\mathbb{R}^n)$ and $\rm BMO_{\Delta_N}(\mathbb{R}^n)$.  Interestingly, we show that these spaces behave in an analogous fashion as the standard Hardy space $H^1(\mathbb{R}^n)$ and $\rm BMO(\mathbb{R}^n)$.

We first explicitly compute the Riesz transforms $R_N=\nabla \Delta_N^{-\frac{1}{2}}$ associated to the Neumann Laplacian.  Because of the close connection between the Laplacian and Neumann Laplacian, we find in Proposition \ref{p:RieszKernel} that the Riesz transforms associated to the Neumann Laplacian are given by an additive perturbation of the standard Riesz transforms.

Our first main result shows that, similar to the classical Hardy space, the space $H^1_{\Delta_N}(\mathbb{R}^n)$ can be characterized by
the radial and non-tangential maximal functions,
by the Riesz transforms,  and by atoms, all of which are defined in terms of the Neumann Laplacian $\Delta_N$.
To be more precise, we denote by $H^1_{\Delta_N,max}(\mathbb{R}^n)$ the Hardy space defined via the radial maximal function associated with $\Delta_N$, and analogously by  $H^1_{\Delta_N,*}(\mathbb{R}^n)$, $H^1_{\Delta_N,Riesz}(\mathbb{R}^n)$ and $H^1_{\Delta_N,atom}(\mathbb{R}^n)$ the Hardy spaces via non-tangential maximal functions,
Riesz transforms  and  atoms, respectively. Then we have the following characterizations.

\begin{theorem}
\label{t:theorem characterization}
Let all notation be the same as above. We have
$$ H^1_{\Delta_N}(\mathbb{R}^n) =
H^1_{\Delta_N,max}(\mathbb{R}^n)=H^1_{\Delta_N,*}(\mathbb{R}^n)=H^1_{\Delta_N,Riesz}(\mathbb{R}^n) =H^1_{\Delta_N,atom}(\mathbb{R}^n)
$$
and with equivalent norms
\begin{eqnarray*}
\|f\|_{H^1_{\Delta_N}(\mathbb{R}^n)}&\thickapprox&
\|f\|_{H^1_{\Delta_N,max}(\mathbb{R}^n)}\thickapprox \|f\|_{H^1_{\Delta_N,Riesz}(\mathbb{R}^n)}
\thickapprox\|f\|_{H^1_{\Delta_N,atom}(\mathbb{R}^n)}\\
&\thickapprox&\|f_{+,e}\|_{H^1(\mathbb{R}^n)}+\|f_{-,e}\|_{H^1(\mathbb{R}^n)}.
\end{eqnarray*}
Here $f_{\pm,e}$ is the even extension of the restriction of $f$ from $\mathbb{R}^n_{\pm}$.  Namely, $f\in H^1_{\Delta_N}(\mathbb{R}^n)$ if and only if
$f_{+,e}\in H^1(\mathbb{R}^n)$ and $f_{-,e}\in H^1(\mathbb{R}^n)$.
\end{theorem}
For more details on these Hardy spaces and the norms we refer to Section \ref{s:H1BMO}.  We also obtain a Fefferman--Stein decomposition of ${\rm BMO}_{\Delta_N}(\mathbb{R}^n)$ in terms of the action of the Riesz transforms associated to the Neumann Laplacian on $L^\infty(\mathbb{R}^n)$ functions. %  , see Corollary \ref{c:FS}.
\begin{cor}
\label{c:FS}
The following are equivalent for a function $b$:
\begin{itemize}
\item[(i)] $b\in {\rm BMO}_{\Delta_N}(\mathbb{R}^n)$;
\item[(ii)] There exists $b_0, b_1,\ldots, b_n\in L^\infty(\mathbb{R}^n)$ such that $b=b_0+\sum_{j=1}^n R_{N, j}^{\ast}b_j$, where
$R_{N, j}^{\ast}$ is the adjoint operator of $R_{N, j}$.
\end{itemize}
\end{cor}

We then further show the connection between ${\rm BMO}_{\Delta_N}(\mathbb{R}^n)$, $H^1_{\Delta_N}(\mathbb{R}^n)$, commutators of functions in ${\rm BMO}_{\Delta_N}(\mathbb{R}^n)$ and Riesz transforms $R_N$ relative to $\Delta_N$, and a weak factorization of the space $H^1_{\Delta_N}(\mathbb{R}^n)$.  In particular, our second main result is the following theorem.

%\color{red}
%There is a slight issue in the statement of the theorems below.  We are using the vector of Riesz transforms in both statements, instead of simply one of the components.  Maybe we should let $1\leq l\neq n$ and define $\Pi_l$ and $\mathcal{C}_{\Delta_N,l}$ is the obvious ways and state the Theorems for all $l$?  Or, we need some comments about duality with vector-valued functions.
%\color{black}

\begin{theorem}
\label{weakfactorization}
For $1\leq l\leq n$, let $\Pi_l(h,g)=h\cdot R^*_{N,l}(g)-g\cdot R_{N,l}(h)$, where $R_{N,l}={\partial\over\partial x_l} \Delta_N^{-\frac{1}{2}}$ is the $l$-th Riesz transform associated to the Neumann Laplacian and $R_{N,l}^*$ is the adjoint operator of $R_{N,l}$.  Then for any $f\in H^1_{\Delta_{N}}(\mathbb{R}^n)$ there exists sequences $\{\lambda_j^k\}\in \ell^1$ and functions $g_j^{k}, h_j^{k}\in L^\infty(\mathbb{R}^n)$ with compact supports such that $f=\sum_{k=1}^\infty\sum_{j=1}^{\infty} \lambda_{j}^{k}\,\Pi_l(g_j^{k}, h_j^{k})$.  Moreover, we have that:
$$
\left\Vert f\right\Vert_{H^1_{\Delta_{N}}(\mathbb{R}^n)} \approx \inf\left\{\sum_{k=1}^\infty \sum_{j=1}^{\infty} \left\vert \lambda_j^k\right\vert \left\Vert g_j^k\right\Vert_{L^2(\mathbb{R}^n)}\left\Vert h_j^k\right\Vert_{L^2(\mathbb{R}^n)}:\ \ f=\sum_{k=1}^\infty\sum_{j=1}^{\infty}\lambda_j^k \,\Pi_l(g_j, h_j) \right\}.
$$
\end{theorem}

We then obtain the following new characterization of ${\rm BMO}_{\Delta_N}(\mathbb{R}^n)$ in terms of the commutators with the Riesz transforms associated to $\Delta_N$.

\begin{theorem}
\label{c:bmo}
Suppose $b\in \cup_{p\geq 1} L^p_{loc}(\mathbb R^n)$.

If $b$ is in ${\rm BMO}_{\Delta_N}(\mathbb{R}^n)$, then for $1\leq l\leq n$, the commutator
$$[b,R_{N,l}](f)(x)= b(x)R_{N,l}(f)(x) - R_{N,l}(bf)(x) $$
is a bounded map on $L^2(\mathbb{R}^n)$, with operator norm
$$
\|[b,R_{N,l}]:L^2(\mathbb{R}^n)\to L^2(\mathbb{R}^n)\| \leq C\|b\|_{\rm BMO_{\Delta_N}(\mathbb{R}^n)}.
$$

Conversely,  for $1\leq l\leq n$, if $[b,R_{N,l}]$ are bounded on $L^2(\mathbb{R}^n)$ then $b$ is in ${\rm BMO}_{\Delta_N}(\mathbb{R}^n)$ and $ \|b\|_{{\rm BMO}_{\Delta_N}(\mathbb{R}^n)} \leq C \| [b,R_{N,l}] :L^2(\mathbb{R}^n)\to L^2(\mathbb{R}^n)\|$.
\end{theorem}
We point out that Theorem \ref{weakfactorization} and Theorem \ref{c:bmo} can be extended to work for $L^p(\mathbb{R}^n)$ when $1<p<\infty$.  %However, we do not explore this extension at this time.

For $0<\alpha<n$, the fractional operator $\Delta_N^{-\alpha/2}$ of the operator $\Delta_N$ is defined by
$$ \Delta_N^{-\alpha/2}f(x) = {1\over \Gamma(\alpha/2)} \int_0^\infty e^{-t\Delta_N}(f)(x) {dt\over t^{1-\alpha/2}}. $$
\begin{theorem}
\label{c:bmo 1}
If $b$ is in ${\rm BMO}_{\Delta_N}(\mathbb{R}^n)$, then for $1<\alpha < n$, the commutator
$$[b,\Delta_N^{-\alpha/2}](f)(x)= b(x)\Delta_N^{-\alpha/2}(f)(x) - \Delta_N^{-\alpha/2}(bf)(x) $$
is a bounded map from $L^p(\mathbb{R}^n)$ to $L^q(\mathbb{R}^n)$ with operator norm
$$
\|[b,\Delta_N^{-\alpha/2}]:L^p(\mathbb{R}^n)\to L^q(\mathbb{R}^n)\| \leq C\|b\|_{{\rm BMO}_{\Delta_N}(\mathbb{R}^n)},
$$
where $1<p<{n\over \alpha}$ and ${1\over q} = {1\over p} - {\alpha\over n}$.

%Conversely, for $1<\alpha < n$, if $[b,\Delta_N^{-\alpha/2}]$ is bounded from $L^p(\mathbb{R}^n)$ to $L^q(\mathbb{R}^n)$, then $b$ is in $\rm BMO_{\Delta_N}(\mathbb{R}^n)$ and $ \|b\|_{\rm BMO_{\Delta_N}(\mathbb{R}^n)} \leq C \|[b,\Delta_N^{-\alpha/2}]:L^p(\mathbb{R}^n)\to L^q(\mathbb{R}^n)\|$.
\end{theorem}

The paper is organized as follows. In Section \ref{s:NeumannRiesz}, we collect the background for the Neumann Laplacian and the associated Riesz transforms.  In Section \ref{s:H1BMO} the related Hardy and $\rm BMO$ spaces associated to $\Delta_N$ are studied and their basic properties are collected.  In particular, we demonstrate a collection of equivalent norms for $H^1_{\Delta_N}(\mathbb{R}^n)$, Theorem \ref{theorem characterization}, and show the Fefferman-Stein decomposition of ${\rm BMO}_{\Delta_N}(\mathbb{R}^n)$ holds, Corollary \ref{c:FS}.  Finally, in Section \ref{s:factorization} we provide the proof of Theorems \ref{weakfactorization} and \ref{c:bmo}.  Throughout this paper, the letter ``$C$" will denote, possibly different, constants  that are independent of the essential variables.

\section{The Neumann Laplacian and the Associated Riesz Kernels}
\setcounter{equation}{0}
\label{s:NeumannRiesz}

We now recall some notation and basic facts introduced in \cite{DDSY}*{Section 2}. For any subset $A\subset \mathbb{R}^n$ and a function $f: \mathbb{R}^n\rightarrow\mathbb{C}$ by $f|_A$ we denote the restriction of $f$ to $A$.   Next we set $\mathbb{R}^n_+=\{(x',x_n)\in \mathbb{R}^n: x'=(x_1,\ldots,x_{n-1})\in \mathbb{R}^{n-1},x_n>0\}$.
For any function $f$ on
$\mathbb{R}^n$, we set
$$ f_+=f|_{\mathbb{R}^n_+}\ \ \ {\rm and}\ \ \ f_-=f|_{\mathbb{R}^n_-}. $$
 For any $x=(x',x_n)\in\mathbb{R}^n$ we set $\widetilde{x}=(x',-x_n)$. If $f$ is any function defined on $\mathbb{R}^n_+$, its even extension defined on $\mathbb{R}^n$ is
\begin{align}\label{f e}
f_e(x)=f(x),\ {\rm if }\ x\in \mathbb{R}^n_+;\ \ f_e(x)=f(\widetilde{x}),\ {\rm if }\ x\in \mathbb{R}^n_-.
\end{align}

\subsection{The Neumann Laplacian}

We denote by $\Delta_n$ the Laplacian on $\mathbb{R}^n$. Next we recall the Neumann Laplacian on $\mathbb{R}^n_+$ and $\mathbb{R}^n_-$.

Consider the Neumann problem on the half line $(0,\infty)$ (see \cite[(7), page 59 in Section 3.1]{S}):
\begin{align}\label{Neumann}
\left\{
\begin{array}{lcc}
 w_t-w_{xx}  =0 & {\rm for\ }  0<x<\infty, 0<t<\infty, \\
 w(x,0)=\phi(x), &  \\
 w_x(0,t)=0. &
\end{array}
\right.
\end{align}
Denote this corresponding Laplacian by $\Delta_{1,N_+}$. According to \cite[(7), Section 3.1]{S}, we see that
$$ w(x,t) = e^{-t\Delta_{1,N_+}}(\phi)(x). $$
 For $n>1$, we write $\mathbb{R}_+^n= \mathbb{R}^{n-1}\times \mathbb{R}_+$. And we define the Neumann Laplacian on $\mathbb{R}^n_+$ by
$$ \Delta_{n,N_+} = \Delta_{n-1} + \Delta_{1,N_+}, $$
where $\Delta_{n-1}$ is the  Laplacian on $\mathbb{R}^{n-1}$ and $\Delta_{1,N_+}$ is the Laplacian corresponding to \eqref{Neumann}.  Similarly we can define Neumann Laplacian  $\Delta_{n,N_-}$ on $\mathbb{R}^n_-$.

%\color{red}
%Is this defintion is equivalent to the other one we talked about?
%\color{black}

%Define the even extension of $\phi$ as
%$$
%\phi_e(x):=
%\left\{
%\begin{array}{ccc}
%\phi(x) & x\geq0\\
% \phi(-x) & x<0.
%\end{array}
%\right.
%$$
%From the reflection (see \cite[Section 3.1]{S}) we have
%$$
%\left\{
%\begin{array}{ccc}
% u_t-u_{xx}=0, &  \\
% u(x,0)=\phi_e(x), &  \\
% u_x(0,t)=0 &
%\end{array}
%\right.
%$$
%for the whole real line $-\infty<x<\infty$ and $0<t<\infty$.

%and by $\Delta_{n,N_+}$ (and $\Delta_{n,N_-}$) we denote the Neumann Laplacian on $\mathbb{R}^n_+$ (and on $\mathbb{R}^n_-$ respectively).

In the remainder of the paper, we skip the index $n$,   we denote by $\Delta$ the Laplacian on $\mathbb{R}^n$, denote the Neumann Laplacian on $\mathbb{R}^n_+$ by $\Delta_{N_+}$, and
Neumann Laplacian on $\mathbb{R}^n_-$ by~$\Delta_{N_-}$.

%\color{red}Is there a differential representation of $\Delta_N$?  Can we include this?  Or at least a definition?\color{black}

The Laplacian and Neumann Laplacian $\Delta_{N_{\pm}}$ are positive definite self-adjoint operators. By the spectral theorem one can define the semigroups generated by these operators $\{\exp(-t\Delta), t\geq 0\}$ and $\{\exp(-t\Delta_{N_\pm}), t\geq 0\}$. By $p_t(x,y)$, $p_{t,\Delta_{N_+}}(x,y)$ and $p_{t,\Delta_{N_-}}(x,y)$ we denote the heat kernels corresponding to the semigroups generated by $\Delta$, $\Delta_{N_+}$ and $\Delta_{N_-}$,
respectively.  Then we have
\begin{eqnarray*}
p_{t}(x,y) & = & \frac{1}{(4\pi t)^{\frac{n}{2}}}e^{-\frac{|x-y|^2}{4t}}.
\end{eqnarray*}
From the reflection method (see \cite[(9), page 60 in Section 3.1]{S}), we get
\begin{eqnarray*}
p_{t,\Delta_{N_+}}(x,y) & =& \frac{1}{(4\pi t)^{\frac{n}{2}}}e^{-\frac{|x'-y'|^2}{4t}}\left( e^{-\frac{|x_n-y_n|^2}{4t}}+e^{-\frac{|x_n+y_n|^2}{4t}} \right), \ \ x,y\in \mathbb{R}^n_+;\\
p_{t,\Delta_{N_-}}(x,y) & =& \frac{1}{(4\pi t)^{\frac{n}{2}}}e^{-\frac{|x'-y'|^2}{4t}}\left( e^{-\frac{|x_n-y_n|^2}{4t}}+e^{-\frac{|x_n+y_n|^2}{4t}} \right), \ \ x,y\in \mathbb{R}^n_-.
\end{eqnarray*}

For any function $f$ on $\mathbb{R}^n_+$, we have $$\exp(-t\Delta_{N_+})f(x)=\exp(-t\Delta)f_e(x)$$ for all $t\geq0$ and $x\in \mathbb{R}^n_+$. Similarly, for any function $f$ on $\mathbb{R}^n_-$, $$\exp(-t\Delta_{N_-})f(x)=\exp(-t\Delta)f_e(x)$$ for all $t\geq0$ and $x\in \mathbb{R}^n_-$.

Now let $\Delta_N$ be the uniquely determined unbounded operator acting on $L^2(\mathbb{R}^n)$ such that
\begin{align}\label{Delta N}
 (\Delta_Nf)_+=\Delta_{N_+}f_+ \ \ \ {\rm and}\ \ \ (\Delta_Nf)_-=\Delta_{N_-}f_-
\end{align}
for all $f: \mathbb{R}^n\rightarrow \mathbb{R}$ such that $f_+\in W^{1,2}(\mathbb{R}^n_+)$ and $f_-\in W^{1,2}(\mathbb{R}^n_-)$. Then
$\Delta_N$ is a positive self-adjoint operator and
$$\ (\exp(-t\Delta_N)f)_+=\exp(-t\Delta_{N_+})f_+\ \ \ {\rm and}\ \ \
(\exp(-t\Delta_N)f)_-=\exp(-t\Delta_{N_-})f_-.$$
The heat kernel of $\exp(-t\Delta_N)$, denoted by $p_{t,\Delta_N}(x,y)$, is then given as:
\begin{eqnarray}
\label{def-of-ptN}
p_{t,\Delta_N}(x,y)=\frac{1}{(4\pi t)^{\frac{n}{2}}}e^{-\frac{|x'-y'|^2}{4t}}\big( e^{-\frac{|x_n-y_n|^2}{4t}}+e^{-\frac{|x_n+y_n|^2}{4t}} \big)H(x_ny_n),
\end{eqnarray}
where $H: \mathbb{R}\rightarrow\{0,1\}$ is the Heaviside function given by
\begin{equation}
\label{e:hvyside}
H(t)=0,\ \ {\rm if}\ t<0;\ \ \ \ \ H(t)=1,\ \ {\rm if}\ t\geq0.
\end{equation}

Let us note that

\begin{itemize}
\item[$(\alpha)$] All the operators $\Delta, \Delta_{N_+}, \Delta_{N_-}, $ and  $\Delta_{N}$ are self-adjoint and they generate bounded analytic positive semigroups acting on all $L^p(\mathbb{R}^n)$ spaces for $1\leq p\leq \infty$;

\item[$(\beta)$] Suppose that $p_{t,L}(x,y)$ is the kernel corresponding to the semigroup generated by one of the operators $L$ listed in $(\alpha)$. Then the kernel $p_{t,L}(x,y)$ satisfies Gaussian bounds:
\begin{align}
\label{size of Neumann heat kernel}
 |p_{t,L}(x,y)|\leq \frac{C}{t^{\frac{n}{2}}} e^{-c\frac{|x-y|^2}{t}},
\end{align}
for all $x,y\in \Omega$, where $\Omega=\mathbb{R}^n$ for $\Delta, \Delta_{N}$; $\Omega=\mathbb{R}^n_+$ for $\Delta_{N_+}$ and $\Omega=\mathbb{R}^n_-$ for $\Delta_{N_-}$.
\end{itemize}

Next we consider the smoothness property of the heat kernel for $\Delta_N$, $\Delta_{N_+}$, and $\Delta_{N_-}$.

\begin{prop} Suppose that $L$ is one of the operators $ \Delta_{N_+}$, $ \Delta_{N_-}$ and $\Delta_{N}$. Then for $x,x',y\in\mathbb{R}^n_+$ $($or $\in\mathbb{R}^n_-$$)$ with
 $|x-x'|\leq \frac{1}{2}|x-y|$, we have
\begin{align}\label{smooth x}
|p_{t,L}(x,y)- p_{t,L}(x',y)| \leq C\frac{|x-x'|}{(\sqrt{t}+|x-y|)}\frac{\sqrt{t}}{(\sqrt{t}+|x-y|)^{n+1}};
\end{align}
symmetrically, for $x,y,y'\in\mathbb{R}^n_+$ $($or $\in\mathbb{R}^n_-$$)$ with
$|y-y'|\leq \frac{1}{2}|x-y|$, we have
\begin{align}\label{smooth y}
|p_{t,L}(x,y)- p_{t,L}(x,y')| \leq C\frac{|y-y'|}{(\sqrt{t}+|x-y|)}\frac{\sqrt{t}}{(\sqrt{t}+|x-y|)^{n+1}}.
\end{align}
\end{prop}

\begin{proof}
Suppose $x,y\in\mathbb{R}^n_+$. Then
for $i=1,\ldots,n-1$, we have
\begin{align*}
\frac{\partial}{\partial x_i} p_{t,\Delta_{N_+}}(x,y) =  -\frac{(x_i-y_i)}{2t} \frac{1}{(4\pi t)^{\frac{n}{2}}}e^{-\frac{|x'-y'|^2}{4t}}\left( e^{-\frac{|x_n-y_n|^2}{4t}}+e^{-\frac{|x_n+y_n|^2}{4t}} \right).
\end{align*}
Moreover,
\begin{align*}
\frac{\partial}{\partial x_n} p_{t,\Delta_{N_+}}(x,y) =   -\frac{1}{(4\pi t)^{\frac{n}{2}}}e^{-\frac{|x'-y'|^2}{4t}}\left( e^{-\frac{|x_n-y_n|^2}{4t}} \frac{(x_n-y_n)}{2t}+e^{-\frac{|x_n+y_n|^2}{4t}} \frac{(x_n+y_n)}{2t}\right).
\end{align*}
Then we obtain that
\begin{align*}
\left |\nabla_x{  p_{t,\Delta_{N_+}}(x,y) }\right|^2&=
\sum_{i=1}^{n-1}\left |{\partial\over\partial x_i} p_{t,\Delta_{N_+}}(x,y)\right|^2 +  \left|{\partial\over\partial x_n} p_{t,\Delta_{N_+}}(x,y) \right|^2\\
&\leq \sum_{i=1}^{n-1} { (x_i-y_i)^2\over 4t^2 } {1\over {(4\pi t)^{n}}}e^{-{{|x'-y'|^2}\over {2t}}}  \left( e^{-{{|x_n-y_n|^2}\over
{4t}}}+   e^{-{{|x_n+y_n|^2}\over {4t}}} \right)^2\\
&\quad+ 2 {1\over {(4\pi t)^{n}}}e^{-{{|x'-y'|^2}\over {2t}}} \left( e^{-\frac{|x_n-y_n|^2}{4t}} \frac{(x_n-y_n)}{2t} \right)^2\\
&\quad+ 2 {1\over {(4\pi t)^{n}}}e^{-{{|x'-y'|^2}\over {2t}}} \left( e^{-\frac{|x_n+y_n|^2}{4t}} \frac{(x_n+y_n)}{2t}
\right)^2\\
&\leq C\sum_{i=1}^{n} { (x_i-y_i)^2\over t^2 } {1\over {(4\pi t)^{n}}}e^{-{{|x-y|^2}\over {2t}}}+2 {1\over {(4\pi t)^{n}}}e^{-{{|x'-y'|^2}\over {2t}}} \left( e^{-\frac{|x_n+y_n|^2}{4t}} \frac{(x_n+y_n)}{2t}
\right)^2\\
&\leq C { |x-y|^2\over t^2 } {1\over {(4\pi t)^{n}}}e^{-{{|x-y|^2}\over {2t}}} +2 {1\over {(4\pi t)^{n}}}e^{-{{|x-y|^2}\over {4t}}} \left( e^{-\frac{|x_n+y_n|^2}{8t}} \frac{(x_n+y_n)}{2t}
\right)^2\\
&\leq C {t\over (t+ |x-y|^2)^{n+2}}.
\end{align*}
Hence, it is easy to verify that
\begin{align*}
|\nabla_x p_{t,\Delta_{N_+}}(x,y)| \leq C\frac{\sqrt{t}}{(\sqrt{t}+|x-y|)^{n+2}}
\end{align*}
and similarly we can obtain that
\begin{align*}
|\nabla_y p_{t,\Delta_{N_+}}(x,y)| \leq C\frac{\sqrt{t}}{(\sqrt{t}+|x-y|)^{n+2}},
\end{align*}
which implies that
\begin{align*}
|p_{t,\Delta_{N_+}}(x,y)- p_{t,\Delta_{N_+}}(x',y)| \leq C\frac{|x-x'|}{(\sqrt{t}+|x-y|)}\frac{\sqrt{t}}{(\sqrt{t}+|x-y|)^{n+1}}
\end{align*}
for $x,x',y\in\mathbb{R}^n_+$ with $|x-x'|\leq \frac{1}{2}|x-y|$, and
\begin{align*}
|p_{t,\Delta_{N_+}}(x,y)- p_{t,\Delta_{N_+}}(x,y')| \leq C\frac{|y-y'|}{(\sqrt{t}+|x-y|)}\frac{\sqrt{t}}{(\sqrt{t}+|x-y|)^{n+1}}
\end{align*}
for $x,x',y\in\mathbb{R}^n_+$ with $|y-y'|\leq \frac{1}{2}|x-y|$.

We can obtain similar estimates for the heat semigroup of $ \Delta_{N_-}$ and $ \Delta_{N}$.
\end{proof}

\subsection{The Riesz Kernels Associated to the Neumann Laplacian}

A fundamental object in our study are the Riesz transforms associated to the Neumann Laplacian.  Recall that the Riesz transforms associated to the Neumann Laplacian are given by: $R_N= \nabla \Delta_N^{-\frac{1}{2}}$.
We collect the formula for these kernels in the following proposition.

\begin{prop}
\label{p:RieszKernel}
Denote %by $R_{j}(x,y)$ the kernel of the classical $j$-th Riesz transform and 
by $R_{N,j}(x,y)$ the kernel of the $j$-th Riesz transform $\frac{\partial}{\partial x_j} \Delta_N^{-\frac{1}{2}}$ of $\Delta_N$.  Then for $1\leq j\leq n-1$ and for $x,y\in\mathbb{R}^n_+$ we have:
$$
R_{N,j}(x,y)= - C_n \bigg( {x_j-y_j\over |x-y|^{n+1}} +   \frac{x_j-y_j}{(|x'-y'|^2+|x_n+y_n|^2)^{\frac{n+1}{2}}}\bigg)
$$
and
$$
R_{N,n}(x,y)= - C_n \bigg( {x_j-y_j\over |x-y|^{n+1}} +   \frac{x_n+y_n}{(|x'-y'|^2+|x_n+y_n|^2)^{\frac{n+1}{2}}}\bigg),
$$
where $C_n=\frac{\Gamma\big(\frac{n+1}{2}\big)}{(\pi )^{\frac{n+1}{2}}}  $.

%with
%$$
%K_{N,j}(x,y):=
%C_n\left\{
%\begin{array}{ccl}
% \frac{x_j-y_j}{(|x'-y'|^2+|x_n+y_n|^2)^{\frac{n+1}{2}}} & : & 1\leq j\leq n-1\\
% \frac{x_n+y_n}{(|x'-y'|^2+|x_n+y_n|^2)^{\frac{n+1}{2}}} & : &  j=n,
%\end{array}
%\right.
%$$
%with $C_n= \frac{2^{n-2}}{(4\pi )^{\frac{n}{2}}}(1-n) \Gamma\big(\frac{n}{2}-1\big)$.  
Similar expressions also hold for $R_{N,j}(x,y) $, $j=1,\ldots,n$, when $x,y\in\mathbb{R}^n_-$.
\end{prop}
\begin{proof}
Working from the definition of the square root of $\Delta_N$, i.e.,
$$ \Delta_N^{-{1\over2}}= {1\over \Gamma({1\over2})} \int_0^\infty e^{-t\Delta_N} {dt\over \sqrt{t}}, $$ 
we have that for $1\leq j\leq n-1$:
\begin{align*}
R_{N,j}(x,y)&= {1\over \Gamma({1\over2})}\frac{\partial}{\partial x_j}  \int_0^\infty  p_{t,\Delta_{N}}(x,y) \,\frac{dt}{\sqrt{t}}\\
&= {1\over \Gamma({1\over2})} \frac{\partial}{\partial x_j} \left( \int_0^\infty  \frac{1}{(4\pi t)^{\frac{n}{
2}}}e^{-\frac{|x-y|^2}{4t}} \, \frac{dt}{\sqrt{t}} +\int_0^\infty  \frac{1}{(4\pi t)^{\frac{n}{
2}}}e^{-\frac{|x'-y'|^2}{4t}}e^{-\frac{|x_n+y_n|^2}{4t}} \,  \frac{dt}{\sqrt{t}}  \right)\\
&=- \frac{\Gamma\Big(\frac{n+1}{2}\Big)}{(\pi )^{\frac{n+1}{2}}}  \bigg( \frac{x_j-y_j}{|x-y|^{n+1}} +  \frac{x_j-y_j}{(|x'-y'|^2+|x_n+y_n|^2)^{\frac{n+1}{2}}} \bigg).
\end{align*}

For $j=n$ and for $x,y\in\mathbb{R}^n_+$ we again observe:
\begin{align*}
R_{N,n}(x,y)&= {\sqrt{\pi}\over2}\frac{\partial}{\partial x_n}  \int_0^\infty  p_{t,\Delta_{N}}(x,y) \,\frac{dt}{\sqrt{t}}\\
&=-\frac{\Gamma\Big(\frac{n+1}{2}\Big)}{(\pi )^{\frac{n+1}{2}}}   \bigg( \frac{x_n-y_n}{|x-y|^{n+1}} +  \frac{x_n+y_n}{(|x'-y'|^2+|x_n+y_n|^2)^{\frac{n+1}{2}}} \bigg).
\end{align*}
\end{proof}

We next make the observation that kernels $R_{N,j}(x,y)$ are Calder\'on--Zygmund kernels.
\begin{prop}
Denote by $R_N(x,y)$ the kernel of the vector of Riesz transforms $\nabla \Delta_N^{-\frac{1}{2}}$.  Then:
\begin{align}\label{Riesz kernel}
R_N(x,y)=\big( R_{N,1}(x,y),\ldots, R_{N,n}(x,y) \big)H(x_ny_n),
\end{align}
with $H(t)$ the Heavyside function defined in \eqref{e:hvyside}.  Moreover, we have that
\begin{align*}
|R_{N}(x,y)|\leq C_n  \frac{1}{|x-y|^{n}},
\end{align*}
and
\begin{align*}
|R_{N}(x,y)-R_{N}(x_0,y)|+|R_{N}(y,x)-R_{N}(y,x_0)|\leq C\frac{|x-x_0|}{|x-y|^{n+1}}
\end{align*}
for $x,x_0,y\in\mathbb{R}^n_+$ $($or $x,x_0,y\in\mathbb{R}^n_-)$
  with $|x-x_0|  \leq \frac{1}{2} |x-y|$.
\end{prop}

\begin{proof}

We first claim that for $j=1,\ldots,n$, and $x,y\in\mathbb{R}^n_+$ (or $x,y\in\mathbb{R}^n_-$)
\begin{align*}
|R_{N,j}(x,y)|\leq C_n  \frac{1}{|x-y|^{n}}.
\end{align*}
In fact, from Proposition \ref{p:RieszKernel}, it is direct that for $1\leq j\leq n-1$,
$$ \frac{|x_j-y_j|}{(|x'-y'|^2+|x_n+y_n|^2)^{\frac{n+1}{2}}} \leq  \frac{|x_j-y_j|}{(|x'-y'|^2+|x_n-y_n|^2)^{\frac{n+1}{2}}} \leq  \frac{1}{|x-y|^{n}}$$
and for $j=n$,
$$ \frac{|x_n+y_n|}{(|x'-y'|^2+|x_n+y_n|^2)^{\frac{n+1}{2}}} \leq  \frac{1}{(|x'-y'|^2+|x_n+y_n|^2)^{\frac{n}{2}}} \leq  \frac{1}{|x-y|^{n}},$$
where we use the fact that $x,x_0,y\in\mathbb{R}^n_+$ (or $x,x_0,y\in\mathbb{R}^n_-$) and hence $x_j+y_j > |x_j-y_j|$ for $1\leq j\leq n$.

%Next we consider the smoothness property of $R_{N,j}(x,y)$. Suppose $j=1,\ldots,n-1$, and $x,y\in\mathbb{R}^n_+$ (or $x,y\in\mathbb{R}^n_-$).  Then it is easy to verify that
%\begin{align*}
%\frac{\partial}{\partial x_i} K_{N,j}(x,y) = -(n+1)C_n \frac{x_j-y_j}{(|x'-y'|^2+|x_n+y_n|^2)^{\frac{n+1}{2}+1}} \cdot(x_i-y_i)
%\end{align*}
%for  $i=1,\ldots,n-1$ and $i\not=j$. When $i=j$, we have
%\begin{align*}
%\frac{\partial}{\partial x_j} K_{N,j}(x,y) = C_n \frac{1}{(|x'-y'|^2+|x_n+y_n|^2)^{\frac{n+1}{2}}} - C_n (n+1)  \frac{ (x_j-y_j)^2}{(|x'-y'|^2+|x_n+y_n|^2)^{\frac{n+3}{2}}}.
%\end{align*}
%Moreover, we have
%\begin{align*}
%\frac{\partial}{\partial x_n} K_{N,j}(x,y) = -(n+1)C_n \frac{x_n+y_n}{(|x'-y'|^2+|x_n+y_n|^2)^{\frac{n+1}{2}+1}} \cdot(x_n+y_n).
%\end{align*}
%Then we get that for $i=1,\ldots,n$,
%$$ \left|\frac{\partial}{\partial x_i} K_{N,j}(x,y) \right|  \leq \frac{C}{|x-y|^{n+1}},$$
%which, together with the smoothness property of the classical Riesz kernel,  gives that
Similarly, by considering the estimates for the terms ${\partial\over \partial x_j}R_{N,j}(x,y)$ and ${\partial\over \partial y_j}R_{N,j}(x,y)$, 
we obtain that
\begin{align*}
|R_{N,j}(x,y)-R_{N,j}(x_0,y)|+|R_{N,j}(y,x)-R_{N,j}(y,x_0)|\leq C\frac{|x-x_0|}{|x-y|^{n+1}}
\end{align*}
for $x,x_0,y\in\mathbb{R}^n_+$ (or $x,x_0,y\in\mathbb{R}^n_-$).
  with $|x-x_0|  \leq \frac{1}{2} |x-y|$.
\end{proof}

%We point out that

%and
%\begin{align*}
%|K_{N,n}(x,y)|&\leq C_n  {1 \over |x-y|^{n}} + C_n {|x_n+y_n|\over (|x'-y'|^2+|x_n+y_n|^2)^{{n+1\over2}}}\\
%&\leq C_n{1 \over |x-y|^{n}} + C_n {1\over (|x'-y'|^2+|x_n+y_n|^2)^{{n\over2}}}\\
%&\leq C_n  {1 \over |x-y|^{n}},
%\end{align*}

\subsection{The Kernels of Fractional operators Associated to the Neumann Laplacian}
For $0<\alpha<n$, denote by $K(x,y)$ the kernel of the classical fractional operator $\Delta^{-\alpha/2}$, which is defined by
$$ \Delta^{-\alpha/2}f(x) = {1\over \Gamma(\alpha/2)} \int_0^\infty e^{-t\Delta}(f)(x) {dt\over t^{1-\alpha/2}}. $$
We know that
$$ K(x,y)={C_{n,\alpha}\over |x-y|^{n-\alpha}}, $$
where $C_{n,\alpha}= {\Gamma({n\over2}-{\alpha\over2}) \over \Gamma({\alpha\over2})}  {1\over \pi^{n\over2} 2^\alpha}$.  It is well known that when $b \in {\rm BMO}(\mathbb{R}^n)$, the commutator $[b,\Delta^{-\alpha/2}]$ is bounded from $L^p(\mathbb{R}^n)$ to $L^q(\mathbb{R}^n)$ for $1<p<n/\alpha$ and $1/q=1/p-\alpha/n$. See \cite{Chanillo}.

\begin{prop}
\label{p:Fractional Kernel}
Denote by $K_N(x,y)$ the kernel of the fractional operator $\Delta_N^{-\alpha/2}$. Then $x,y\in\mathbb{R}^n_+$ we have:
$$
K_{N}(x,y)= K(x,y)+\tilde{K}_{N}(x,y)
$$
with
$$
\tilde{K}_{N}(x,y):=
C_{n,\alpha}\frac{1}{(|x'-y'|^2+|x_n+y_n|^2)^{{\frac{n}{2}}-{\alpha\over2}}}. 
$$
  Similar expressions for $K_{N}(x,y) $ when $x,y\in\mathbb{R}^n_-$ also hold.
\end{prop}
\begin{proof}
For $x,y\in\mathbb{R}^n_+$, working from the fraction of the square root of $\Delta_N$ we have that:
\begin{align*}
&K_{N}(x,y)\\
&=  {1\over \Gamma(\alpha/2)}  \int_0^\infty  p_{t,\Delta_{N}}(x,y) \,{dt\over t^{1-\alpha/2}}\\
&={1\over \Gamma(\alpha/2)}  \int_0^\infty  \frac{1}{(4\pi t)^{\frac{n}{
2}}}e^{-\frac{|x-y|^2}{4t}} \, {dt\over t^{1-\alpha/2}} +{1\over \Gamma(\alpha/2)} \int_0^\infty  \frac{1}{(4\pi t)^{\frac{n}{
2}}}e^{-\frac{|x'-y'|^2}{4t}}e^{-\frac{|x_n+y_n|^2}{4t}} \,  {dt\over t^{1-\alpha/2}}  \\
&= C_{n,\alpha}  \Big(\frac{1}{|x-y|^{n-\alpha}}   +    \frac{1}{(|x'-y'|^2+|x_n+y_n|^2)^{{\frac{n}{2}}-{\alpha\over2}}}  \Big)\\
%&=\frac{2^{n-2}}{(4\pi )^{\frac{n}{2}}} (1-n)\Gamma\Big(\frac{n}{2}-1\Big) \bigg( \frac{x_j-y_j}{|x-y|^{n+1}} +  \frac{x_j-y_j}{(|x'-y'|^2+|x_n+y_n|^2)^{\frac{n+1}{2}}} \bigg)\\
&= K(x,y) + \tilde{K}_N(x,y).
\end{align*}
where we set $$ \tilde{K}_{N}(x,y)= C_{n,\alpha}\frac{1}{(|x'-y'|^2+|x_n+y_n|^2)^{{\frac{n}{2}}-{\alpha\over2}}}.$$
%with $C_n= \frac{2^{n-2}}{(4\pi )^{\frac{n}{2}}} \Gamma\big(\frac{n}{2}-1\big) $.  %This proves the formula for $1\leq j\leq n-1$.
\end{proof}

\section{Characterization and Properties of \texorpdfstring{$H^1_{\Delta_N}(\mathbb{R}^n)$}{Neumann H1} and \texorpdfstring{${\rm BMO}_{\Delta_N}(\mathbb{R}^n)$}{Neumann BMO}}
\setcounter{equation}{0}
\label{s:H1BMO}

\subsection{Fundamental Properties of \texorpdfstring{${\rm BMO}_{\Delta_N}(\mathbb{R}^n)$}{Neumann BMO}}

We now recall the definition and  some fundamental properties of ${\rm BMO}_{\Delta_N}(\mathbb{R}^n)$ from \cite{DDSY}.

Define $$ \mathcal{M}=\left\{ f\in L^1_{loc}(\mathbb{R}^n): \ \exists d>0\ s.t.\ \int_{\mathbb{R}^n} {|f(x)|^2\over 1+|x|^{n+d}} \,dx <\infty  \right\}. $$

\begin{definition}[\cite{DDSY}*{Definition 2.2}]
We say that $f\in\mathcal{M}$ is of bounded mean oscillation
associated with $\Delta_N$, abbreviated as ${\rm BMO}_{\Delta_N}(\mathbb{R}^n)$, if
\begin{align}\label{BMON norm}
 \|f\|_{{\rm BMO}_{\Delta_N}(\mathbb{R}^n)} = \sup_{B(y,r)} {1\over |B(y,r)|}\int_{B(y,r)}\left|f(x)-\exp(-r^2\Delta_N)f(x)\right|\,dx<\infty,
\end{align}
where the supremum is taken over all balls $B(y, r)$ in $\mathbb{R}^n$. The smallest bound for which \eqref{BMON norm} is satisfied is then taken to be the norm of f in this space, and is denoted by $ \|f\|_{{\rm BMO}_{\Delta_N}(\mathbb{R}^n)} $.
\end{definition}

\begin{definition}[\cite{DDSY}*{Definition 2.1}]
A function $f$ on $\mathbb{R}^n_+$ is said to be in ${\rm BMO}_r(\mathbb{R}^n_+)$ if there exists $F \in {\rm BMO}(\mathbb{R}^n)$ such that $F|_{\mathbb{R}^n_+} = f$. If $f \in {\rm BMO}_r(\mathbb{R}^n_+)$, then we set
$$  \|f\|_{{\rm BMO}_r(\mathbb{R}^n_+)}=\inf\left\{ \|F\|_{\rm BMO(\mathbb{R}^n)}:  F|_{\mathbb{R}^n_+} = f \right\}.  $$
\end{definition}

\begin{definition}[\cite{DDSY}*{Page 270}]
For any function $f\in L^1_{loc}(\mathbb{R}^n_+)$, define
$$ \|f\|_{{\rm BMO}_e(\mathbb{R}^n_+)}= \|f_e\|_{{\rm BMO}(\mathbb{R}^n)}, $$
where $f_e$ is defined in \eqref{f e}. We denote by ${\rm BMO}_e(\mathbb{R}^n_+)$ the corresponding Banach space.
\end{definition}

Similarly we can define the spaces ${\rm BMO}_r(\mathbb{R}^n_-)$ and ${\rm BMO}_e(\mathbb{R}^n_-)$.

\begin{prop}[\cite{DDSY}*{Proposition 3.1}]
The spaces ${\rm BMO}_r(\mathbb{R}^n_+)$ and ${\rm BMO}_e(\mathbb{R}^n_+)$ coincide, and their norms are equivalent. Similar result holds for ${\rm BMO}_r(\mathbb{R}^n_-)$ and ${\rm BMO}_e(\mathbb{R}^n_-)$.
\end{prop}

\begin{prop}[\cite{DDSY}*{Proposition 4.2}]
The Neumann {\rm BMO} space ${\rm BMO}_{\Delta_N}(\mathbb{R}^n)$ can be described in the following way:
$${\rm BMO}_{\Delta_N}(\mathbb{R}^n)=\left\{f\in\mathcal{M}: \ f_+\in {\rm BMO}_r(\mathbb{R}^n_+) {\rm\ and\ }  f_-\in {\rm BMO}_r(\mathbb{R}^n_-)  \right\}.$$
\end{prop}

As a consequence of the results from \cite{DDSY} listed above, we obtain that $f\in {\rm BMO}_{\Delta_N}(\mathbb{R}^n)$ if and only if $f_{+,e}, f_{-,e}\in {\rm BMO}(\mathbb{R}^n)$.   A final key fact that plays a role in our analysis is the duality between ${\rm BMO}_{\Delta_N}(\mathbb{R}^n)$ and $H^1_{\Delta_N}(\mathbb{R}^n)$.
\begin{prop}[\cite{DDSY}*{Corollary 4.3}]
\label{p:duality}
The dual space of $H^1_{\Delta_N}(\mathbb{R}^n)$ is ${\rm BMO}_{\Delta_N}(\mathbb{R}^n)$.
\end{prop}

%\color{red}PLEASE TYPE IN THE BASICS WE NEED FROM \cite{DDSY}.\color{black}

\subsection{Properties of \texorpdfstring{$H^1_{\Delta_N}(\mathbb{R}^n)$}{Neumann H1}}

In this subsection, we provide a deeper study of the space $H^1(\mathbb{R}^n)$.

We first provide several equivalent characterizations of $H_{\Delta_N}^1(\mathbb{R}^n)$.  To do so, we need the following definitions of the Hardy space associated to $\Delta_N$ in terms of the radial maximal function, the non-tangential maximal function, the Riesz transforms, and atoms.  As one might expect, these definitions all turn out to be equivalent as shown below in Theorem \ref{theorem characterization}.

\begin{definition}\label{def radial max H1}
We define $H^1_{\Delta_N,max}(\mathbb{R}^n)=\left\{ f\in L^1(\mathbb{R}^n):  f_{\Delta_N}^+ \in L^1(\mathbb{R}^n) \right\} $
with the norm $\|f\|_{H^1_{\Delta_N,max}(\mathbb{R}^n)} =\|f_{\Delta_N}^+\|_{L^1(\mathbb{R}^n)}$, where
$ f_{\Delta_N}^+(x) = \sup\limits_{t>0}\left |\exp(-t^2\Delta_N)f(x)\right|$.
\end{definition}
\begin{definition}\label{def non-tangential max H1}
We define $H^1_{\Delta_N,*}(\mathbb{R}^n)=\left\{ f\in L^1(\mathbb{R}^n):  f_{\Delta_N}^* \in L^1(\mathbb{R}^n) \right\} $
with the norm $\|f\|_{H^1_{\Delta_N,*}(\mathbb{R}^n)} =\|f_{\Delta_N}^*\|_{L^1(\mathbb{R}^n)}$, where
$ f_{\Delta_N}^*(x) = \sup\limits_{|x-y|<t}\left |\exp(-t^2\Delta_N)f(y)\right|$.
\end{definition}

%\color{red}
%Again, there is an issue with vector-valued versus scalar valued.  Maybe change the definition to $\partial_{x_l} \Delta_N^{-1/2} f\in L^1$ for all $1\leq l\leq n$?
%\color{black}

\begin{definition}\label{def Riesz H1}
We define $$H^1_{\Delta_N,Riesz}(\mathbb{R}^n)=\left\{ f\in L^1(\mathbb{R}^n):  {\partial\over\partial x_l}\Delta_N^{-\frac{1}{2}} f \in L^1(\mathbb{R}^n)\  {\rm for \ }  1\leq l\leq n\right\} $$
with the norm $\|f\|_{H^1_{\Delta_N,Riesz}(\mathbb{R}^n)} =\left\Vert f\right\Vert_{L^1(\mathbb{R}^n)}+\sum_{l=1}^n \left\|\frac{\partial}{\partial x_l}\Delta_N^{-\frac{1}{2}} f \right\|_{L^1(\mathbb{R}^n)}$.
\end{definition}
Next we define the atoms for $H^1_{\Delta_N,max}(\mathbb{R}^n)$, which we adapt from a very recent result of Song and Yan \cite{SY}.
\begin{definition}\label{def atom}
Given $M\in \mathbb{N}$. We say that a function $a(x) \in L^\infty(\mathbb{R}^n)$ is an $H^1_{\Delta_N,max}(\mathbb{R}^n)$-atom, if there exist a function $b$ in the domain of $\Delta_N^M$ and a ball $B\subset \mathbb{R}^n$ such that
\begin{itemize}
\item[(i)] $a=\Delta_N^Mb$;
\item[(ii)] $\operatorname{supp}\,\Delta_N^kb \subset B$, $k=0,1,\ldots, M$;
\item[(iii)] $\| (r_B^2\Delta_N)^kb \|_{L^\infty(\mathbb{R}^n)} \leq r_B^{2M} |B|^{-1}$,  $k=0,1,\ldots, M$.
\end{itemize}
\end{definition}
\begin{definition}\label{def atom H1}
We say that $f=\sum_j\lambda_ja_j$ is an atomic representation of $f$ if $\{\lambda_j\}\in \ell^1$, each $a_j$ is an $H^1_{\Delta_N,max}(\mathbb{R}^n)$ atom, and the sum converges in $L^2(\mathbb{R}^n)$. Set
$$ \widetilde{H}^1_{\Delta_N,atom}(\mathbb{R}^n)=\left\{ f\in L^2(\mathbb{R}^n): \ f\ {\rm has\ an\ atomic\ representation} \right\} $$
with the norm $\|f\|_{\widetilde{H}^1_{\Delta_N,atom}(\mathbb{R}^n)}$ given by
$$ \inf\bigg\{ \sum_j|\lambda_j|:\ f=\sum_j\lambda_ja_j \ {\rm is\ an\ atomic\ representation} \bigg\}.  $$
The space $H^1_{\Delta_N,atom}(\mathbb{R}^n)$ is defined as the completion of $\widetilde{H}^1_{\Delta_N,atom}(\mathbb{R}^n)$ with respect to this norm.
\end{definition}

We now collection the equivalence of all these definitions and moreover provide a link between $H^1(\mathbb{R}^n)$ and $H^1_{\Delta_N}(\mathbb{R}^n)$.
\begin{theorem}\label{theorem characterization}
Let all the notation be as above. Then,
$$ H^1_{\Delta_N}(\mathbb{R}^n) =
H^1_{\Delta_N,max}(\mathbb{R}^n)=H^1_{\Delta_N,*}(\mathbb{R}^n)=H^1_{\Delta_N,Riesz}(\mathbb{R}^n) =H^1_{\Delta_N,atom}(\mathbb{R}^n)
$$
and they have equivalent norms
\begin{eqnarray*}
\|f\|_{H^1_{\Delta_N}(\mathbb{R}^n)}&\thickapprox&
\|f\|_{H^1_{\Delta_N,max}(\mathbb{R}^n)}\thickapprox \|f\|_{H^1_{\Delta_N,Riesz}(\mathbb{R}^n)}
\thickapprox\|f\|_{H^1_{\Delta_N,atom}(\mathbb{R}^n)}\\
&\thickapprox&\|f_{+,e}\|_{H^1(\mathbb{R}^n)}+\|f_{-,e}\|_{H^1(\mathbb{R}^n)}.
\end{eqnarray*}
Namely, $f\in H^1_{\Delta_N}(\mathbb{R}^n)$ if and only if
$f_{+,e}\in H^1(\mathbb{R}^n)$ and $f_{-,e}\in H^1(\mathbb{R}^n)$.
\end{theorem}

\begin{proof}

We recall that the Hardy space associated with $\Delta_N$ is defined as the set of functions  $\left\{f\in L^1(\mathbb{R}^n):
\|S_{\Delta_N}(f)\|_{L^1(\mathbb{R}^n)}<\infty\right\}$ in the norm
of $\|f\|_{H^1_{\Delta_N}}=\|S_{\Delta_N}(f)\|_{L^1(\mathbb{R}^n)}$, where $
S_{\Delta_N}(f)(x)= \Big( \int_0^{\infty}\!\!\!\int_{|y-x|<t}
 |Q_{t^2}f(y)|^2\,   \frac{dydt}{t^{n+1}}  \Big)^{\frac{1}{2}}
$, and  $Q_{t^2}=t^2\Delta_N\exp(-t^2\Delta_N)$.

We now consider the operator
$Q_{t}=t\Delta_N\exp(-t\Delta_N) = -t {d\over dt} \exp(-t\Delta_N)$  for any $t> 0$ (see \cite{DY2}*{(3.5) in Section 3.1}). %Note that from
%
%As
%introduced in $\S2.3$ \color{red}There is no section 2.3.  What Section are we referring to?\color{black},
Then we have
$$Q_{t^2}f(x)=t^2\Delta_N\exp(-t^2\Delta_N)f(x)=\int_{\mathbb{R}^n}-\frac{t}{2} \frac{\partial}{\partial t}
p_{t^2,\Delta_N}(x,y)f(y)\,dy.$$ From the definition of
$p_{t,\Delta_N}(x,y)$, see \eqref{def-of-ptN}, we have that for
any $x\in \mathbb{R}^n_+$,
\begin{eqnarray*}
t^2\Delta_N\exp(-t^2\Delta_N)f(x)&=&\int_{\mathbb{R}^n_+}-\frac{t}{2} \frac{\partial}{\partial t} p_{t^2,\Delta_N}(x,y)f_+(y)\,dy\\
&=&\int_{\mathbb{R}^n}-\frac{t}{2} \frac{\partial}{\partial t}
p_{t^2}(x,y)f_{+,e}(y)\,dy\\
&=&t^2\Delta\exp(-t^2\Delta)f_{+,e}(x).
\end{eqnarray*}
Similarly, for any $x\in \mathbb{R}^n_-$, we have $
t^2\Delta_N\exp(-t^2\Delta_N)f(x)=t^2\Delta\exp(-t^2\Delta)f_{-,e}(x).
$

Moreover, by a change of variable,
\begin{eqnarray}\label{e4.4 qt reflection}
&&t^2\Delta_N\exp(-t^2\Delta_N)f(x)=-t^2\Delta\exp(-t^2\Delta)f_{+,e}(\widetilde{x})\
\ {\rm for\ any\ }t>0, \ x\in \mathbb{R}^n_+;\\
&&t^2\Delta_N\exp(-t^2\Delta_N)f(x)=-t^2\Delta\exp(-t^2\Delta)f_{-,e}(\widetilde{x})\
\ {\rm for\ any\ }t>0, \ x\in \mathbb{R}^n_-.\nonumber
\end{eqnarray}
Then from \eqref{e4.4 qt reflection} we have
\begin{align*}
S_{\Delta_N}(f)(x)^2
&=\int_0^\infty \int_{ |x-y|<t, y\in\mathbb{R}_+^n }  |t^2\Delta_N\exp(-t^2\Delta_N)f(y)|^2 \,\frac{dydt}{t^n}\\
  &\hskip.5cm+
   \int_0^\infty \int_{ |x-y|<t, y\in\mathbb{R}_-^n }  |t^2\Delta_N\exp(-t^2\Delta_N)f(y)|^2 \,\frac{dydt}{t^n} \\
&=\int_0^\infty \int_{ |x-y|<t, y\in\mathbb{R}_+^n }  |t^2\Delta\exp(-t^2\Delta)f_{+,e}(y)|^2\, \frac{dydt}{t^n}\\
  &\hskip.5cm+
   \int_0^\infty \int_{ |x-y|<t, y\in\mathbb{R}_-^n }  |t^2\Delta\exp(-t^2\Delta)f_{-,e}(y)|^2\, \frac{dydt}{t^n}\\
&=\frac{1}{2}\bigg(\int_0^\infty \int_{ |x-y|<t }  |t^2\Delta\exp(-t^2\Delta)f_{+,e}(y)|^2\, \frac{dydt}{t^n}\\
  &\hskip.5cm+
   \int_0^\infty \int_{ |x-y|<t }  |t^2\Delta\exp(-t^2\Delta)f_{-,e}(y)|^2\, \frac{dydt}{t^n}\bigg),
\end{align*}
which implies that $ S_{\Delta_N}(f)(x)\leq \frac{\sqrt{2}}{2} \Big( S(f_{+,e})(x)+S(f_{-,e})(x) \Big) $. Conversely,
\begin{align*}
S(f_{+,e})(x)^2
&=\int_0^\infty \int_{ |x-y|<t }  |t^2\Delta\exp(-t^2\Delta)f_{+,e}(y)|^2\, \frac{dydt}{t^n}\\
&=2 \int_0^\infty \int_{ |x-y|<t, y\in \mathbb{R}^n_+ }  |t^2\Delta\exp(-t^2\Delta)f_{+,e}(y)|^2 \frac{dydt}{t^n}\\
&\leq 2 S_{\Delta_N}(f)(x)^2.
\end{align*}
Similarly we have $S(f_{-,e})(x)^2\leq 2 S_{\Delta_N}(f)(x)^2$.  Hence, we obtain that $ S(f_{+,e})(x)+S(f_{-,e})(x) \leq   2\sqrt{2}S_{\Delta_N}(f)(x) $. As a consequence, we have
\begin{align}\label{e4.5}
\|f\|_{H^1_{\Delta_N}(\mathbb{R}^n)}&=\int_{\mathbb{R}^n}  |S_{\Delta_N}(f)(x)|\,dx\\
&\approx \int_{\mathbb{R}^n}  |S(f_{+,e})(x)|\,dx+ \int_{\mathbb{R}^n}  |S(f_{-,e})(x)|\,dx\nonumber\\
 &= \|f_{+,e}\|_{H^1(\mathbb{R}^n)}+\|f_{-,e}\|_{H^1(\mathbb{R}^n)}. \nonumber
\end{align}

% Then, the
%result in Section 3 implies that
%$$H_{\Delta_D}^1(\mathbb{R}^n)=H_{\Delta_D,G}^1(\mathbb{R}^n) $$
%and that they have equivalent norms, namely
%$\|f\|_{H^1_{\Delta_D}(\mathbb{R}^n)}\thickapprox\|f\|_{H^1_{\Delta_D,G}(\mathbb{R}^n)}$.
%

%

Next we turn to $H^1_{\Delta_D,max}(\mathbb{R}^n)$.  From \eqref{def-of-ptN} we can see that for any $t\geq0$ and $x\in \mathbb{R}^n_+$,
\begin{align*}
\exp(-t^2\Delta_N)f(x)&=\int_{\mathbb{R}^n}p_{t^2,\Delta_N}(x,y)f(y)\,dy=
\int_{\mathbb{R}^n_+}p_{t^2,\Delta_N}(x,y)f_+(y)\,dy\\
&=
\int_{\mathbb{R}^n}p_{t^2}(x,y)f_{+,e}(y)\,dy=\exp(-t^2\Delta)f_{+,e}(x).
\end{align*}
Similarly, $\exp(-t^2\Delta_N)f(x)=\exp(-t^2\Delta)f_{-,e}(x)$
for any $t\geq0$ and $x\in \mathbb{R}^n_-$. Thus,
\begin{eqnarray}\label{e4.1}
&&\sup_{t>0}|\exp(-t^2\Delta_N)f(x)|=\sup_{t>0}|\exp(-t^2\Delta)f_{+,e}(x)|\
\ {\rm for\ any\ }\ x\in\mathbb{R}^n_+;\\
&&\sup_{t>0}|\exp(-t^2\Delta_N)f(x)|=\sup_{t>0}|\exp(-t^2\Delta)f_{-,e}(x)|\
\ {\rm for\ any\ }\ x\in\mathbb{R}^n_-.\nonumber
\end{eqnarray}
Again, by a change of variable, we have that
\begin{eqnarray}\label{e4.2}
&&\exp(-t^2\Delta_N)f(x)=-\exp(-t^2\Delta)f_{+,e}(\widetilde{x})\
\ {\rm for\ any\ }t>0, \ x\in \mathbb{R}^n_+;\\
&&\exp(-t^2\Delta_N)f(x)=-\exp(-t^2\Delta)f_{-,e}(\widetilde{x})\
\ {\rm for\ any\ }t>0, \ x\in \mathbb{R}^n_-.\nonumber
\end{eqnarray}
Then, for any $f\in H^1_{\Delta_N, max}(\mathbb{R}^n)$, from
\eqref{e4.1} and \eqref{e4.2} we can obtain that
\begin{eqnarray}\label{e4.3}
&&\|f\|_{H^1_{\Delta_N,
max}(\mathbb{R}^n)}=\int_{\mathbb{R}^n_+}|f^+_{\Delta_N}(x)|\,dx+\int_{\mathbb{R}^n_-}|f^+_{\Delta_N}(x)|\,dx\\
&&=\int_{\mathbb{R}^n_+}\sup_{t>0}|\exp(-t^2\Delta_N)f(x)|\,dx+
\int_{\mathbb{R}^n_-}\sup_{t>0}|\exp(-t^2\Delta_N)f(x)|dx\nonumber\\
&&=\int_{\mathbb{R}^n_+}\sup_{t>0}|\exp(-t^2\Delta)f_{+,e}(x)|\,dx+
\int_{\mathbb{R}^n_-}\sup_{t>0}|\exp(-t^2\Delta)f_{-,e}(x)|\,dx\nonumber\\
&&=\frac{1}{2}\Big(\int_{\mathbb{R}^n}\sup_{t>0}|\exp(-t^2\Delta)f_{+,e}(x)|\,dx+
\int_{\mathbb{R}^n}\sup_{t>0}|\exp(-t^2\Delta)f_{-,e}(x)|\,dx\Big)\nonumber\\
&&=\frac{1}{2}\Big(\|(f_{+,e})^+\|_{L^1(\mathbb{R}^n)}+\|(f_{-,e})^+\|_{L^1(\mathbb{R}^n)}\Big)\nonumber\\
&&=\frac{1}{2}\Big(\|f_{+,e}\|_{H^1(\mathbb{R}^n)}+\|f_{-,e}\|_{H^1(\mathbb{R}^n)}\Big),\nonumber
\end{eqnarray}
where $f^+(x)=\sup\limits_{t>0}|p_{t^2}*f(x)|$ is the classical
maximal function as defined in (3) in \S2.4. Thus \eqref{e4.3}
yields that $f\in H^1_{\Delta_N, max}(\mathbb{R}^n)$ if and only
if $f_{+,e}\in H^1(\mathbb{R}^n)$ and $f_{-,e}\in
H^1(\mathbb{R}^n)$.

We now consider the Hardy space $H^1_{\Delta_N,
*}(\mathbb{R}^n)$ via the non-tangential maximal
function. Note that
\begin{eqnarray*}
f_{\Delta_N}^*(x)&=&\sup_{|x-y|<t}|\exp(-t^2\Delta_N)f(y)|\\
&\leq& \sup_{|x-y|<t, y\in \mathbb{R}^n_+}
|\exp(-t^2\Delta_N)f(y)|+\sup_{|x-y|<t, y\in \mathbb{R}^n_-}
|\exp(-t^2\Delta_N)f(y)|\\
&\leq& \sup_{|x-y|<t, y\in \mathbb{R}^n_+}
|\exp(-t^2\Delta)f_{+,e}(y)|+\sup_{|x-y|<t, y\in \mathbb{R}^n_-}
|\exp(-t^2\Delta)f_{-,e}(y)|\\
&\leq& \sup_{|x-y|<t} |\exp(-t^2\Delta)f_{+,e}(y)|+\sup_{|x-y|<t}
|\exp(-t^2\Delta)f_{-,e}(y)|\\
 &=& (f_{+,e})^*(x)+(f_{-,e})^*(x),
\end{eqnarray*}
where $f^*(x)=\sup\limits_{|x-y|<t}|p_{t^2}*f(y)|$ is the
classical non-tangential maximal function. Hence $\|f_{\Delta_N}^*(x)\|_{L^1(\mathbb{R}^n)}\leq
\|(f_{+,e})^*\|_{L^1(\mathbb{R}^n)}+\|(f_{-,e})^*\|_{L^1(\mathbb{R}^n)}.$ Moreover, we
have
\begin{eqnarray*}
(f_{+,e})^*(x)&=&\sup_{|x-y|<t}|\exp(-t^2\Delta)f_{+,e}(y)|\\
&\leq &\sup_{|x-y|<t, y\in \mathbb{R}^n_+}
|\exp(-t^2\Delta)f_{+,e}(y)|+\sup_{|x-y|<t, y\in \mathbb{R}^n_-}
|\exp(-t^2\Delta)f_{+,e}(y)|\\
&\leq & 2\sup_{|x-y|<t}|\exp(-t^2\Delta_N)f(y)|\\
&\leq& 2f_{\Delta_N}^*(x).
\end{eqnarray*}
Thus, $\|(f_{+,e})^*\|_{L^1(\mathbb{R}^n)}\leq
2\|f_{\Delta_N}^*\|_{L^1(\mathbb{R}^n)}$. Similarly we obtain 
$\|(f_{-,e})^*\|_1\leq
2\|f_{\Delta_N}^*\|_{L^1(\mathbb{R}^n)}$. This implies that
\begin{eqnarray}\label{e4.4}
\|f_{\Delta_N}^*\|_1\thickapprox
\|(f_{+,e})^*\|_{L^1(\mathbb{R}^n)}+\|(f_{-,e})^*\|_{L^1(\mathbb{R}^n)}.
\end{eqnarray}
Thus, \eqref{e4.4} yields that $f\in H^1_{\Delta_N,*}(\mathbb{R}^n)$ if and only if $f_{+,e}\in H^1(\mathbb{R}^n)$ and $f_{-,e}\in H^1(\mathbb{R}^n)$.

As for the Riesz transform characterization of the Hardy space $H^1_{\Delta_N}(\mathbb{R}^n)$, it suffices to note that when $x\in\mathbb{R}^n_+$,
\begin{align*}
  \nabla\Delta_N^{-\frac{1}{2}} f(x) &= \int_{\mathbb{R}^n} K_N(x,y) f(y)\,dy = \int_{\mathbb{R}^n_+} R_N(x,y) f_+(y)\,dy
  =    \int_{\mathbb{R}^n} R(x,y) f_{+,e}(y)\,dy\\
  &=  \nabla\Delta^{-\frac{1}{2}} f_{+,e}(x)
\end{align*}
and that
when $x\in\mathbb{R}^n_-$,
\begin{align*}
  \nabla\Delta_N^{-\frac{1}{2}} f(x) =  \nabla\Delta^{-\frac{1}{2}} f_{-,e}(x).
\end{align*}
Thus,  $f\in H^1_{\Delta_N,Riesz}(\mathbb{R}^n)$ if and only if $f_{+,e}\in H^1(\mathbb{R}^n)$ and $f_{-,e}\in H^1(\mathbb{R}^n)$.

Finally, for the atomic decomposition, in the recent paper of Song and Yan \cite{SY}, they already obtained that $ H^1_{\Delta_N,*}(\mathbb{R}^n)=H^1_{\Delta_N,atom}(\mathbb{R}^n)
$. See \cite{SY}*{Theorem 1.4} for this fact.
\end{proof}

We now prove the Fefferman--Stein type representation for the space ${\rm BMO}_{\Delta_N}(\mathbb{R}^n)$.

\begin{proof}[Proof of Corollary \ref{c:FS}]
The proof is as in \cite{FS}.  Let $B=\bigoplus_{j=0}^n L^1(\mathbb{R}^n)$ and norm $B$ by $\sum_{j=0}^{n}\left\Vert f_j\right\Vert_{L^1(\mathbb{R}^n)}$.  We have that $B^*=\bigoplus_{j=0}^n L^\infty(\mathbb{R}^n)$.  Let $S$ be the subspace of $B$ given by
$$
S=\left\{(f, R_{N,1}f,\ldots,R_{N,n}f):f\in L^1(\mathbb{R}^n)\right\}.
$$
We have that $S$ is a closed subspace and that $f\to (f, R_{N,1}f,\ldots,R_{N,n}f)$ is a isometry of $H^1_{\Delta_{N}}(\mathbb{R}^n)$ to $S$.  Linear functionals on $S$ and $H^1_{\Delta_N}(\mathbb{R}^n)$ can be identified in an obvious way, hence any continuous linear functional on $H^1_{\Delta_N}(\mathbb{R}^n)$ can be extended by Hahn-Banach to a continuous linear functional on $B$ and can be identified with a vector of functions $(b_0,b_1,\ldots, b_n)$ with each $b_j\in L^\infty(\mathbb{R}^n)$.

We use this conclusion in the following way.  Let $\ell$ be a continuous linear functional on $H^1_{\Delta_N}(\mathbb{R}^n)$.  Then by Proposition \ref{p:duality} there is a function $b\in {\rm BMO}_{\Delta_N}(\mathbb{R}^n)$ so that:
$$
\int_{\mathbb{R}^n} f(x) \overline{b(x)}\,dx=\ell(f).
$$
However, by the discussion above, and by restricting the extended linear functional back to $H^1_{\Delta_N}(\mathbb{R}^n)$ we have for $(f,R_{N,1}f,\ldots,R_{N,n}f)=(f_0,\ldots,f_n)$:
$$
\ell(f)=\sum_{j=0}^n \int_{\mathbb{R}^n} f_j(x) \overline{b_j(x)}\,dx.
$$
Using the definition of the $f_j=R_{N,j}f$ we see that:
$$
\ell(f)= \int_{\mathbb{R}^n} f(x)\overline{\left(b_0(x)+\sum_{j=1}^nR_{N,j}^{\ast}b_j(x)\right)}  \,dx.
$$
This then gives the decomposition that any $b\in \rm BMO_{\Delta_N}(\mathbb{R}^n)$ can be written as:
$$
b=b_0+\sum_{j=1}^nR_{N,j}^{\ast}b_j
$$
with $b_j\in L^\infty(\mathbb{R}^n)$.

For the converse, we simply observe that
 from our Theorem \ref{theorem characterization}, we obtained that $R_N$ maps  $H^1_{\Delta_N}(\mathbb{R}^n)$  to $L^1(\mathbb{R}^n)$.  Hence, the boundedness of the Riesz transform $R_N^{\ast}$ from $L^\infty(\mathbb{R}^n)$ to ${\rm{BMO}}_{\Delta_N}(\mathbb{R}^n)$ follows from duality of $H^1_{\Delta_N}(\mathbb{R}^n)$ with ${\rm{BMO}}_{\Delta_N}(\mathbb{R}^n)$.  We then have that any $b$ that can be written as:
$$
b=b_0+\sum_{j=1}^nR_{N,j}^{\ast}b_j
$$
with $b_j\in L^\infty(\mathbb{R}^n)$ must belong to ${\rm BMO}_{\Delta_N}(\mathbb{R}^n)$.
\end{proof}

We next note that  $H^1_{\Delta_N}(\mathbb{R}^n)$ is a proper subspace of the classical  $H^1(\mathbb{R}^n)$, which was proved by Yan in \cite{Yan}*{Proposition 6.2} from the viewpoint of the semigroup generated by $\Delta_N$.
And we now give a direct proof and provide a specific function $f$ which lies in  $H^1(\mathbb{R}^n)$ but does not belong to $H^1_{\Delta_N}(\mathbb{R}^n)$.  A related claim is made in \cite{DDSY}*{Corollary 4.3}.

\begin{theorem}[\cite{Yan}*{Proposition 6.2}]\label{theorem inclusion}
$ H^1_{\Delta_N}(\mathbb{R}^n)\subsetneq H^1(\mathbb{R}^n). $
\end{theorem}

\begin{proof}
We first show that the containment $ H^1_{\Delta_N}(\mathbb{R}^n)\subset H^1(\mathbb{R}^n)$ holds.  This follows directly from the fact that corresponding $\rm BMO$ spaces norm the $H^1$ spaces, namely that:
$$
\left\Vert f\right\Vert_{H^1_{\Delta_N}(\mathbb{R}^n)}\approx \sup_{\left\Vert b\right\Vert_{{\rm BMO}_{\Delta_N}(\mathbb{R}^n)}\leq 1} \left\vert \left\langle f,b\right\rangle_{L^2(\mathbb{R}^n)}\right\vert.
$$
 An identical statement holds for $H^1(\mathbb{R}^n)$ and $\rm BMO(\mathbb{R}^n)$.  As shown in \cite{DDSY},  $  {\rm BMO}({\mathbb R}^n) \subsetneq
{\rm BMO}_{\Delta_N}({\mathbb R}^n)$, and so we have
$$
\left\Vert f\right\Vert_{H^1(\mathbb{R}^n)}\approx \sup_{\left\Vert b\right\Vert_{{\rm BMO}(\mathbb{R}^n)}\leq 1} \left\vert \left\langle f,b\right\rangle_{L^2(\mathbb{R}^n)}\right\vert\leq \sup_{\left\Vert b\right\Vert_{{\rm BMO}_{\Delta_N}(\mathbb{R}^n)}\leq 1} \left\vert \left\langle f,b\right\rangle_{L^2(\mathbb{R}^n)}\right\vert\approx \left\Vert f\right\Vert_{H^1_{\Delta_N}(\mathbb{R}^n)}.
$$
This gives the containment, $ H^1_{\Delta_N}(\mathbb{R}^n)\subset H^1(\mathbb{R}^n)$.

We now show that there exists a function $f\in H^1(\mathbb{R}^n)$ but $f\not\in H^1_{\Delta_N}(\mathbb{R}^n)$.
For the sake of simplicity, we just consider the example in dimension 1.

Define $$ f(x) := \frac{\chi_{ [0,1] }(x)}{\sqrt{2}}  -   \frac{\chi_{ [-1,0) }(x)}{\sqrt{2}}. $$
It is easy to see that $f(x)$ is supported in $[-1,1]$, and $\int_{\mathbb{R}} f(x)\,dx=0$. Moreover, we have
\begin{align*}
\|f\|_{L^2(\mathbb{R})} =1.
\end{align*}
These implies that $f$ is an atom of $H^1(\mathbb{R})$, which shows that $f\in H^1(\mathbb{R})$.  From the definition of $f$,  we obtain that $f_+(x) =  \frac{\chi_{ [0,1] }(x)}{\sqrt{2}}$, and the even extension is
$$ f_{+,e}(x)  =  \frac{\chi_{ [-1,1] }(x)}{\sqrt{2}}.  $$
But, then it is immediate that $f_{+,e}\notin H^1(\mathbb{R})$ since $\int_{\mathbb{R}^n} f_{+,e}(x)\,dx\not=0$.  One can also prove this by using the equivalent definition of $H^1(\mathbb{R})$ via the radial maximal function.

Similarly we have these estimates for $f_{-,e}$.  Hence, $f_{+,e} \not\in H^1(\mathbb{R})$ and $f_{-,e} \not\in H^1(\mathbb{R})$, which, combining the result in Theorem \ref{theorem characterization},  implies that $f\not\in H^1_{\Delta_N}(\mathbb{R}).$
\end{proof}

Finally, we provide a description of the atoms in $H^1_{\Delta_{N}}(\mathbb{R}^n)$ that connects back to the atom in $H^1(\mathbb{R}^n)$.

\begin{prop}\label{prop cancellation}
Suppose $a(x)$ is an $H^1_{\Delta_N}(\mathbb{R}^n)$-atom supported in $B\subset \mathbb{R}^n$ as in Definition \ref{def atom}. Then we have
\begin{align}\label{cancellation}
 \int_{\mathbb{R}^n}a(x)\,dx=0.
\end{align}
Moreover, if $B\cap \{ x\in \mathbb{R}^n: x_n=0 \}\not=\emptyset$, we denote $B_+= B\cap \mathbb{R}^n_+$ and $B_-= B\cap \mathbb{R}^n_-$. Then we have
\begin{align}\label{partial cancellation}
  \int_{B_+}a(x)\,dx= \int_{B_-}a(x)\,dx=0.
\end{align}
\end{prop}

\begin{proof}
First note that from Theorem \ref{theorem inclusion}, $H^1_{\Delta_N}(\mathbb{R}^n)\subsetneq H^1(\mathbb{R}^n)$. Since $a(x)$ is an $H^1_{\Delta_N}(\mathbb{R}^n)$ atom, we have $a(x)\in H^1(\mathbb{R}^n)$, and hence \eqref{cancellation} holds, where we use \cite{Gra}*{Corollary 6.7.7}.

Second, suppose $B\cap \{ x\in \mathbb{R}^n: x_n=0 \}\not=\emptyset$.  Then we define $a_+(x) = a(x)|_{B_+}$ and $a_-(x) = a(x)|_{B_-}$. Since  $a(x)\in H^1_{\Delta_N}(\mathbb{R}^n)$, from Theorem \ref{theorem characterization} we obtain that both $a_{+,e}(x)$ and $a_{-,e}(x)$ are in $H^1(\mathbb{R}^n)$, which implies that $$ \int_{\mathbb{R}^n} a_{+,e}(x)\,dx=\int_{\mathbb{R}^n} a_{-,e}(x)\,dx=0. $$
Next we claim that $\int_{\mathbb{R}^n} a_{+}(x)dx=0$. In fact,
\begin{align*}
\int_{\mathbb{R}^n} a_{+,e}(x)\,dx = \int_{\mathbb{R}^n_+} a_{+,e}(x)\,dx +\int_{\mathbb{R}^n_-} a_{+,e}(x)\,dx = 2\int_{\mathbb{R}^n_+} a_{+,e}(x)\,dx.
\end{align*}
Hence,  $ \int_{\mathbb{R}^n} a_{+,e}(x)\,dx=0$ implies that  $\int_{\mathbb{R}^n} a_{+}(x)\,dx=0$, i.e., $ \int_{B_+}a(x)\,dx=0$.

Similarly we obtain that  $\int_{B_-}a(x)\,dx=0$. Hence \eqref{partial cancellation} holds.
\end{proof}

\begin{remark}
In \cite{W}, it was asked if a proper subspace of the classical Hardy space exists in which the subspace is characterized by maximal functions. This question was answered positively in \cite{UW}. Our result above, Theorem \ref{theorem characterization},  also gives a proper subspace of the classical Hardy space where the subspace is characterized by radial maximal functions as well as non-tangential maximal functions.
\end{remark}

%Based on the properties of Hardy space \texorpdfstring{$H^1_{\Delta_N}(\mathbb{R}^n)$ and the cancellation of atoms as above, we obtain that
%\begin{prop}
%
%Suppose $a(x)$ is an $H^1_{\Delta_N}(\mathbb{R}^n)$-atom supported in $B\subset \mathbb{R}^n$ as in Definition \ref{def atom}. Then we have
%\begin{align}\label{cancellation}
% \int_{\mathbb{R}^n}a(x)\,dx=0.
%\end{align}
%Moreover, if $B\cap \{ x\in \mathbb{R}^n: x_n=0 \}\not=\emptyset$, we denote $B_+= B\cap \mathbb{R}^n_+$ and $B_-= B\cap \mathbb{R}^n_-$. Then we have
%\begin{align}\label{partial cancellation}
%  \int_{B_+}a(x)\,dx= \int_{B_-}a(x)\,dx=0.
%\end{align}
%\end{prop}
%

%\color{red}
%\color{black}

%\section{Commutator}
%\setcounter{equation}{0}

\section{Weak Factorization of the Hardy space \texorpdfstring{$H^1_{\Delta_N}(\mathbb{R}^n)$}{Neumann H1}}
\setcounter{equation}{0}
\label{s:factorization}

In this section we turn to proving Theorem \ref{weakfactorization}.  There are two parts to this Theorem, and upper and lower bound, and we focus first on the (easier) upper bound.

%\color{red}
%This section as the problem of vector-valued versus scalar valued as well.  Maybe simply changing $\Pi$ to $\Pi_l$ and $\mathcal{C}_{\Delta_N}$ to $\mathcal{C}_{\Delta_N,l}$ and adding in $1\leq l\leq n$ in the statement of the theorems will suffice.  The proofs will have to be modified appropriately as well.
%\color{black}

Recall that, for notational simplicity, we are letting $$\Pi_l(h,g):=h\cdot R^*_{N,l}(g)-g\cdot R_{N,l}(h),$$ where $R_{N,l}={\partial\over\partial x_l} \Delta_N^{-\frac{1}{2}}$ for $1\leq l\leq n$.  We now prove the following theorem.

\begin{theorem}
\label{t:upperbound}
Let $g,h\in L^\infty(\mathbb{R}^n)$ with compact supports.  Then for $1\leq l\leq n$,
$$
\left\Vert \Pi_l(h,g)\right\Vert_{H^1_{\Delta_N}(\mathbb{R}^n)}\leq C \left\Vert g\right\Vert_{L^2(\mathbb{R}^n)}\left\Vert h\right\Vert_{L^2(\mathbb{R}^n)}.
$$
\end{theorem}
This will be an immediate corollary of the following theorem.
\begin{theorem}\label{th upper}
If $b\in {\rm BMO}_{\Delta_N}(\mathbb{R}^n)$, then for $1\leq l\leq n$, the commutator
$$[b,R_{N,l}](f)(x)= b(x)R_{N,l}(f)(x) - R_{N,l}(bf)(x) $$
is a bounded map on $L^2(\mathbb{R}^n)$, with operator norm
$$
\|[b,R_{N,l}]:L^2(\mathbb{R}^n)\to L^2(\mathbb{R}^n)\| \leq C\|b\|_{{\rm BMO}_{\Delta_N}(\mathbb{R}^n)}.
$$
\end{theorem}
\begin{proof}
Suppose $b$ is in ${\rm BMO}_{\Delta_N}(\mathbb{R}^n)$. Then according to \cite{DDSY}*{Proposition 4.2}, we have that
$b_{+,e}\in {\rm BMO}(\mathbb{R}^n)$ and $b_{-,e}\in \rm BMO(\mathbb{R}^n)$, and moreover,
$$ \|b\|_{{\rm BMO}_{\Delta_N}(\mathbb{R}^n)} \approx  \| b_{+,e}\|_{\rm BMO(\mathbb{R}^n)} + \|b_{-,e}\|_{\rm BMO(\mathbb{R}^n)}. $$

For every $f\in L^2(\mathbb{R}^n)$, we have
\begin{align*}
\|  [b,R_{N,l}](f) \|_{L^2(\mathbb{R}^n)}^2 = \int _{\mathbb{R}^n_+} [b,R_{N,l}](f)(x)^2 \,dx+ \int _{\mathbb{R}^n_-} [b,R_{N,l}](f)(x)^2 \,dx=: I+II.
\end{align*}
For the term $I$, note that when $x\in\mathbb{R}^n_+$, we have
\begin{align*}
[b,R_{N,l}](f)(x)&= b(x)R_{N,l}(f)(x) - R_{N,l}(bf)(x)\\
&= b_{+,e}(x) R_l(f_{+,e})(x) - R_l(b_{+,e}f_{+,e})(x)= [b_{+,e},R_l](f_{+,e})(x),
\end{align*}
which implies that
\begin{align*}
I &= \int _{\mathbb{R}^n_+} [b,R_{N,l}](f)(x)^2 \,dx =   \int _{\mathbb{R}^n_+} [b_{+,e},R_l](f_{+,e})(x)^2 \,dx \\
& \leq  \int _{\mathbb{R}^n} [b_{+,e},R_l](f_{+,e})(x)^2 \,dx\\
&\leq C   \| b_{+,e} \|_{\rm BMO(\mathbb{R}^n)}^2   \| f_{+,e} \|_{L^2(\mathbb{R}^n)}^2,
\end{align*}
where $R_l$ is the classical $l$-th Riesz transform ${\partial\over\partial x_l} \Delta^{-{1\over2}}$,

For the last estimate we use the result \cite{CRW}*{Theorem 1}, which applies since we know from Proposition \ref{Riesz kernel} that $R_{N,l}$ is a Calder\'on--Zygmund kernel.  Similarly we can obtain that
\begin{align*}
II
&\leq C   \| b_{-,e} \|_{\rm BMO(\mathbb{R}^n)}^2   \| f_{-,e} \|_{L^2(\mathbb{R}^n)}^2.
\end{align*}
Combining the estimates for $I$ and $II$ above, we obtain that
\begin{align*}
\left\| [b,R_{N,l}](f) \right\|_{L^2(\mathbb{R}^n)}^2 &\leq C   \| b_{+,e} \|_{\rm BMO(\mathbb{R}^n)}^2   \| f_{+,e} \|_{L^2(\mathbb{R}^n)}^2+ C\| b_{-,e} \|_{\rm BMO(\mathbb{R}^n)}^2   \| f_{-,e} \|_{L^2(\mathbb{R}^n)}^2\\
&\leq  C\| b \|_{\rm BMO_{\Delta_N}(\mathbb{R}^n)}^2  \Big( \| f_{+,e} \|_{L^2(\mathbb{R}^n)}^2+ \| f_{-,e} \|_{L^2(\mathbb{R}^n)}^2\Big)\\
&\leq  C\| b \|_{\rm BMO_{\Delta_N}(\mathbb{R}^n)}^2  \| f \|_{L^2(\mathbb{R}^n)}^2,
\end{align*}
which yields that
$
\left\|[b,R_{N,l}] : L^2(\mathbb{R}^n)\to L^2(\mathbb{R}^n)\right\| \leq C\|b\|_{\rm BMO_{\Delta_N}(\mathbb{R}^n)}.$
\end{proof}

\begin{proof}[Proof of Theorem \ref{t:upperbound}]
By the duality result of \cite{DDSY}, stated in Proposition \ref{p:duality}, we know that $H^1_{\Delta_N}(\mathbb{R}^n)^{*}={\rm BMO}_{\Delta_N}(\mathbb{R}^n)$.  A simple duality computation shows for $b\in {\rm BMO}_{\Delta_N}(\mathbb{R}^n)$ and for any $g,h\in L^\infty(\mathbb R^n)$ with compact supports:
\begin{eqnarray*}
\left\langle b, \Pi_l(g,h)\right\rangle_{L^2(\mathbb{R}^n)}  =  \left\langle b, R_{N,l}^*(g)h-R_{N,l}(h)g\right\rangle_{L^2(\mathbb{R}^n)} = \left\langle g,[b,R_{N,l}]h\right\rangle_{L^2(\mathbb{R}^n)}.
\end{eqnarray*}
Thus,  from Theorem \ref{th upper}, we obtain that
\begin{eqnarray*}
\left|\left\langle b, \Pi_l(g,h)\right\rangle_{L^2(\mathbb{R}^n)}\right| \leq C \|b\|_{{\rm BMO}_{\Delta_N}(\mathbb{R}^n)  }\|g\|_{L^2(\mathbb R^n)} \|h\|_{L^2(\mathbb R^n)}.
\end{eqnarray*}
This, together with the duality of $H^1_{\Delta_N}(\mathbb{R}^n)$ with ${\rm BMO}_{\Delta_N}(\mathbb{R}^n)$ shows that
$\Pi_l(g,h)$ is in $H^1_{\Delta_N}(\mathbb{R}^n)$. And then by testing
$\Pi_l(g,h)$ against $b\in {\rm BMO}_{\Delta_N}(\mathbb{R}^n)$ functions, we find:
\begin{eqnarray*}
\left\Vert \Pi_l(g,h)\right\Vert_{H^1_{\Delta_N}(\mathbb{R}^n)} & \approx &  \sup_{\left\Vert b\right\Vert_{\rm BMO_{\Delta_N}(\mathbb{R}^n)}\leq 1}  \left\vert \left\langle \Pi_l(g,h),b\right\rangle_{L^2(\mathbb{R}^n)}\right\vert\\
%& = &  \sup_{\left\Vert b\right\Vert_{\rm BMO_{\Delta_N}(\mathbb{R}^n)}\leq 1} \left\vert \left\langle [b,R_{N,l}](g),h\right\rangle_{L^2(\mathbb{R}^n)}\right\vert\\
%& \leq & \left\Vert g\right\Vert_{L^2(\mathbb{R}^n)}\left\Vert h\right\Vert_{L^2(\mathbb{R}^n)} \sup_{\left\Vert b\right\Vert_{\rm BMO_{\Delta_N}(\mathbb{R}^n)}\leq 1}\left\Vert [b,R_{N,l}] :L^2(\mathbb{R}^n)\to L^2(\mathbb{R}^n)\right\Vert\\
& \leq & C \left\Vert g\right\Vert_{L^2(\mathbb{R}^n)}\left\Vert h\right\Vert_{L^2(\mathbb{R}^n)} \sup_{\left\Vert b\right\Vert_{{\rm BMO}_{\Delta_N}(\mathbb{R}^n)}\leq 1} \left\Vert b\right\Vert_{{\rm BMO}_{\Delta_N}(\mathbb{R}^n)}\\
& \leq & C \left\Vert g\right\Vert_{L^2(\mathbb{R}^n)}\left\Vert h\right\Vert_{L^2(\mathbb{R}^n)}.
\end{eqnarray*}

\end{proof}

\subsection{The Lower Bound in Theorem \ref{weakfactorization}}

The proof of the lower bound is more algorithmic in nature and follows a proof strategy developed by Uchiyama in \cite{U}.  We begin with a fact that will play a prominent role in the algorithm below.  It is a modification of a related fact for the standard Hardy space $H^1(\mathbb{R}^n)$.

\begin{lemma}\label{lemma Hardy}
Suppose $f$ is a function satisfying:  $\int_{\mathbb{R}^n} f(x)\,dx=0$, and $|f(x)|\leq \chi_{B(x_0,1)}(x)+\chi_{B(y_0,1)}(x)$, where $|x_0-y_0|:=M>10$. Then we have
\begin{align}
 \|f\|_{H_{\Delta_N}^1(\mathbb{R}^n)} \leq C_n\log M.
\end{align}
\end{lemma}

\begin{proof}
First note that
\begin{align*}
f_{\Delta_N}^+ (x)&=\sup_{t>0} |e^{-t{\Delta_N}}f(x)|=\sup_{t>0} \left|\int_{\mathbb{R}^n} p_{t,\Delta_N}(x,y)f(y) \,dy\right|\leq  \sup_{t>0} \int_{\mathbb{R}^n} |p_{t,\Delta_N}(x,y)| \,dy \leq C.
\end{align*}
Hence, we obtain that
\begin{align*}
\int_{B(x_0,5)} f_{\Delta_N}^+(x)\,dx + \int_{B(y_0,5)} f_{\Delta_N}^+(x)\,dx\leq C_n.
\end{align*}
Now it suffices to estimate
\begin{align*}
\int_{\mathbb{R}^n\backslash (B(x_0,5) \cup B(y_0,5))} f_{\Delta_N}^+(x)\,dx.
\end{align*}
To see this, we write it as
\begin{align*}
\int_{\mathbb{R}^n\backslash B(x_0, 2M) } f_{\Delta_N}^+(x)\,dx  +\int_{  B(x_0, 2M) \backslash (B(x_0,5) \cup B(y_0,5))} f_{\Delta_N}^+(x)\,dx=: I+II.
\end{align*}
We now estimate the term $I$. First note that from H\"older's regularity  \eqref{smooth y} of the heat kernel $p_{t,\Delta_N}(x,y)$, we have
%\begin{align*}
% |p_t(x,y) - p_t(x,x_0)|\leq C  \Big({\frac{|y-x_0|}{\sqrt{t}}}\Big)^{\sigma_1} \frac{e^{-|x-x_0|^2/ct}}{t^{n/2}}
%\end{align*}
%whenever $|y-x_0|\le \sqrt{t}$ and for any $0<\sigma_1<\sigma_0$. Thus, we have
\begin{align*}
 |p_{t,\Delta_{N}}(x,y) - p_{t,\Delta_{N}}(x,x_0)|\leq C \Big({\frac{|y-x_0|}{  \sqrt{t} +|x-x_0| }}\Big) \frac{ \sqrt{t} }{ (\sqrt{t}+|x-x_0|)^{n+1}  }
 \end{align*}
 for $|y-x_0|< \sqrt{t}$. Moreover,  when $|y-x_0|\geq \sqrt{t}$,  we have
\begin{align*}
 |p_{t,\Delta_{N}}(x,y) - p_{t,\Delta_{N}}(x,x_0)| &\leq |p_{t,\Delta_{N}}(x,y) |+|p_{t,\Delta_{N}}(x,x_0)| \leq C  \frac{e^{-|x-x_0|^2/ct}}{t^{n/2}}  + \frac{e^{-|x-y|^2/ct}}{t^{n/2}}  \\
 &\leq C   \left({\frac{|y-x_0|}{\sqrt{t}}}\right)  \frac{e^{-|x-x_0|^2/ct}}{t^{n/2}}\\
 &\leq C \left({\frac{|y-x_0|}{  \sqrt{t} +|x-x_0| }}\right) \frac{ \sqrt{t} }{ (\sqrt{t}+|x-x_0|)^{n+1}  }.
 \end{align*}

Now note that from the cancellation condition of $f$ and H\"older's regularity of the heat kernel $p_t(x,y)$ as above, we have
\begin{align*}
f_{\Delta_N}^+(x)&=\sup_{t>0} \left|\int_{\mathbb{R}^n} [p_{t,\Delta_{N}}(x,y) - p_{t,\Delta_{N}}(x,x_0)] f(y) \,dy\right|\\
 &\leq C\sup_{t>0} \int_{B(x_0,1)\cup B(y_0,1)}  \left({\frac{|y-x_0|}{  \sqrt{t} +|x-x_0| }}\right) \frac{ \sqrt{t} }{ (\sqrt{t}+|x-x_0|)^{n+1}  }  \,dy \\
&\leq C_n  \frac{|y_0-x_0|}{  |x-x_0|^{n+1} } = C_n \frac{M}{  |x-x_0|^{n+1} }.
\end{align*}
As a consequence, we obtain that
\begin{align*}
I\leq  \int_{\mathbb{R}^n\backslash B(x_0, 2M) } C_n \frac{M}{  |x-x_0|^{n+1} } \,dx \leq C_n.
\end{align*}

We now turn to the term $II$. Note that when $x\in B(x_0, 2M) \backslash (B(x_0,5) \cup B(y_0,5))$,  we have
\begin{align*}
 \left|\int_{\mathbb{R}^n} p_{t,\Delta_{N}}(x,y) f(y) \,dy\right|
 &\leq  \int_{B(x_0,1)}   |p_{t,\Delta_{N}}(x,y)| \,dy +\int_{B(y_0,1)}   |p_{t,\Delta_{N}}(x,y)| \,dy.
\end{align*}
When $t>1$, from the size estimate of the heat kernel $p_{t,\Delta_{N}}(x,y)$, we have
\begin{align*}
 \left|\int_{\mathbb{R}^n} p_{t,\Delta_{N}}(x,y) f(y) \,dy\right|
 &\leq  C\frac{1}{ |x-x_0|^n} + C \frac{1}{|x-y_0|^n}.
\end{align*}
When $t\leq1$, similarly we obtain that
\begin{align*}
 \left|\int_{\mathbb{R}^n} p_{t,\Delta_{N}}(x,y) f(y) \,dy\right|
 &\leq  C\frac{1}{|x-x_0|^{n+1}} + C \frac{1}{|x-y_0|^{n+1}} \leq  C\frac{1}{|x-x_0|^{n}} + C\frac{1}{|x-y_0|^{n}}.
\end{align*}
Thus,
\begin{align*}
II&\leq  \int_{  B(x_0, 2M) \backslash (B(x_0,5) \cup B(y_0,5))} f_{\Delta_N}^+(x)\,dx\\
&\leq C\int_{  B(x_0, 2M) \backslash (B(x_0,5) \cup B(y_0,5))} \frac{1}{|x-x_0|^{n}} +  \frac{1}{|x-y_0|^{n}} \,dx\\
&\leq C_n\log M.
\end{align*}

Combining all the estimates above, we obtain that
\begin{align*}
 \|f\|_{H_{\Delta_{N}}^1(\mathbb{R}^n)}  = \|f_{\Delta_N}^+\|_{L^1(\mathbb{R}^n)} \leq C_n\log M.
\end{align*}
\end{proof}

%\color{red}FROM HERE ON NEEDS TO BE MODIFIED BASED ON THE NEW DEFINITION OF $\Pi(g,h)$.  I believe I have addressed this, please check.\color{black}

Suppose $1\leq l\leq n$.
Ideally, given an $H^1_{\Delta_{N}}(\mathbb{R}^n)$-atom $a$, we would like to find $g,h\in L^2(\mathbb{R}^n)$ such that $\Pi_l(g,h)=a$ pointwise.  While this can't be accomplished in general, the Theorem below shows that it is ``almost'' true.

\begin{theorem}
\label{thm:ApproxFactorization}
Suppose $1\leq l\leq n$. For every $H^1_{\Delta_N}(\mathbb{R}^n)$-atom $a(x)$ and for all $\varepsilon>0$ there exist
a large positive number $M$ and $g,h\in L^\infty(\mathbb{R}^n)$ with compact supports such that:
$$
\left\Vert a-\Pi_l(h,g)\right\Vert_{H^1_{\Delta_{N}}(\mathbb{R}^n)}<\varepsilon
$$
and $\left\Vert g\right\Vert_{L^2(\mathbb{R}^n)}\left\Vert h\right\Vert_{L^2(\mathbb{R}^n)}\leq C M^n$.
\end{theorem}

\begin{proof}

Let $a(x)$ be an $H^1_{\Delta_N}(\mathbb{R}^n)$-atom, supported in $B(x_0,r)$.  We first consider the construction of the bilinear form $\Pi_l(h,g)$ for $1\leq l \leq n-1$ and the approximation 
to $a(x)$.  To begin with, for the ball $B(x_0,r)$, we now consider the following cases:  Case 1: $x_{0,n}\geq 0$; \ \  Case 2: $x_{0,n}<0$.

%Case 3: $B(x_0,r)\cap \mathbb{R}^n_+ \not=\emptyset$ and $B(x_0,r)\cap \mathbb{R}^n_- \not=\emptyset$.

%If 
%Without lost of generality, we assume that $x_0\in\mathbb{R}^n_+$ (If $x_{0,i} $ lies on the $i$-th axis, then we take another fixed point $x'_{0,i} \in \mathbb{R}^n_+$ with $|x_0-x'_0|={r\over 100} $).

We first consider Case 1. To begin with,
fix $\varepsilon>0$.  Choose $M\in [100,\infty)$ sufficiently large so that $ \frac{\log M}{M} <\varepsilon$. Now select $y_0\in\mathbb{R}^n_+$
in the following way:
%
%\smallskip
%(1)  $B(y_0,r)\subset \mathbb{R}^n_+$;
%
\smallskip
%(2) 
%
 for $1\leq i\leq n$, choose $y_{0,i} >0 $ such that  $ y_{0,i}-x_{0,i}=\frac{Mr}{\sqrt{n}}$, where $x_{0,i}$ (reps. $y_{0,i}$) is the $i$th coordinate of $x_{0}$ (reps. $y_{0}$).

%\smallskip
%(3) if $x_{0,n} > \frac{Mr}{\sqrt{n}} $, we choose $y_{0,n}>0$ such that $ x_{0,n}-y_{0,n}=\frac{Mr}{2\sqrt{n}}$,  and if
%$x_{0,n} < \frac{Mr}{\sqrt{n}}$, we choose  $y_{0,n}>0$ such that $ y_{0,n}-x_{0,n}=\frac{Mr}{\sqrt{n}}$.

\smallskip

Note that for this $y_0$, it is clear that $B(y_0,r)\subset \mathbb{R}^n_+$ and 
we have $|x_0-y_0|={Mr}$. Moreover, for any $y\in B(y_0,r)$, we also have
$|x_0-y|>{Mr\over2}$.  We set
\begin{align}\label{gh}
 g(x):=\chi_{B(y_0,r)}(x)\quad{\rm and}\quad h(x):= -\frac{a(x)}{R_{N,l}^*g(x_0)}.  
\end{align}

We first claim that
\begin{align}\label{claim degenerate}
\left|R_{N,l}^*g(x_0) \right|\geq C M^{-n}, \qquad  1\leq l\leq n-1.
\end{align}
%To see this, note that
%\begin{align*}
%R_{N,l}g(x_0) = \int_{B(y_0,r)} R_{N,l}(x_0,y)\,dy. %= \left( \int_{B(y_0,r)} R_{N,1}(x_0,y)\,dy,\ldots,\int_{B(y_0,r)} R_{N,n}(x_0,y)\,dy \right).
%\end{align*}
In fact, for $l=1,\ldots,n-1$, from Proposition \ref{p:RieszKernel}, we have
\begin{align*}
R_{N,l}^*g(x_0) &=\left |\int_{B(y_0,r)}R_{N,l}(y,x_0)\,dy \right| \\
&= C_n \left|\int_{B(y_0,r)}  \left( \frac{y_l-x_{0,l}}{|x_0-y|^{n+1}} +  \frac{y_l-x_{0,l}}{(|x'_0-y'|^2+|x_{0,n}+y_n|^2)^{\frac{n+1}{2}}} \right)\, dy\right|\\
&= C_n\left |y_l-x_{0,l}\right| \left|\int_{B(y_0,r)}  \left( \frac{1}{|x_0-y|^{n+1}} +  \frac{1}{(|x'_0-y'|^2+|x_{0,n}+y_n|^2)^{\frac{n+1}{2}}} \right) \,dy\right|\\
&\geq C Mr \int_{B(y_0,r)}   \frac{1}{|x_0-y|^{n+1}} \,dy \geq CM^{-n}.
\end{align*}
As a consequence, we get that the claim \eqref{claim degenerate} holds.

As for Case 2, we handle it in a symmetric way as follows. Fix $\varepsilon>0$.  Choose $M\in [100,\infty)$ sufficiently large so that $ \frac{\log M}{M} <\varepsilon$. Now select $y_0\in\mathbb{R}^n_+$
in the following way:
%
%\smallskip
%(1)  $B(y_0,r)\subset \mathbb{R}^n_+$;
%
\smallskip
%(2) 
%
 for $1\leq i\leq n$, choose $y_{0,i} >0 $ such that  $ y_{0,i}-x_{0,i}=-\frac{Mr}{\sqrt{n}}$.  Note that for this $y_0$, it is clear that $B(y_0,r)\subset \mathbb{R}^n_-$ and  we have $|x_0-y_0|={Mr}$. Moreover, for any $y\in B(y_0,r)$, we also have $|x_0-y|>{Mr\over2}$.  We now define the functions $g$ and $h$ as in \eqref{gh}, and the following the same estimates, we can obtain that
the claim \eqref{claim degenerate} holds.

%For Case 3, since $B(x_0,r)\cap \mathbb{R}^n_+ \not=\emptyset$ and $B(x_0,r)\cap \mathbb{R}^n_- \not=\emptyset$,
%without lost of generality, we can assume that $x_{0,n}>0$ since the case 
%$x_{0,n}<0$ can be handled symmetrically. If $x_{0,n}=0$, we just choose another $x'_{0,n} = {r\over2}>0$.
%Then we follow the same way as in Case 1 to choose $y_0$, and then $g$ and $h$. And we can further get
%the claim \eqref{claim degenerate} holds for these $g$ and $h$.

%\color{red}
%Should we say anything about "non-degenerate" kernels and this working more generally?
%\color{black}

From the definitions of the functions $g$ and $h$, we obtain that $\operatorname{supp}\,g(x)=B(y_0,r)$ and $\operatorname{supp}\,h(x)=B(x_0,r)$. Moreover, from \eqref{claim degenerate} we obtain that
$$\|g\|_{L^2(\mathbb{R}^n)} \approx r^{\frac{n}{2}}\quad\textnormal{ and }\quad \|h\|_{L^2(\mathbb{R}^n)} =  \frac{1}{|R_{N,l}g(x_0)|} \|a\|_{L^2(\mathbb{R}^n)}\leq C M^n r^{-\frac{n}{2}}. $$
Hence $\|g\|_{L^2(\mathbb{R}^n)} \|h\|_{L^2(\mathbb{R}^n)} \leq CM^n$.  Now write
\begin{align*}
a(x)-\left(h(x)R_{N,l}^*g(x)- g(x)R_{N,l}h(x)\right)&= a(x) \frac{R_{N,l}^*g(x_0)-R_{N,l}^*g(x)}{R_{N,l}^*g(x_0)} - g(x)R_{N,l}h(x)\\
&=: W_1(x)+W_2(x).
\end{align*}
By definition, it is obvious that $W_1(x)$ is supported on $B(x_0,r)$ and $W_2(x)$ is supported on $B(y_0,r)$.

We first turn to $W_1(x)$. For $x\in B(x_0,r)$,
\begin{align*}
|W_1(x)| &= |a(x)|\frac{ |R_{N,l}^*g(x_0) -R_{N,l}^*g(x) |}{R_{N,l}^*g(x_0) }\\
&\leq  C  M^n \|a\|_{L^\infty(\mathbb{R}^n)}  \int_{B(y_0,r)} | R_{N,l}(y,x_0)-R_{N,l}(y,x) | \,dy\\
&\leq  C\frac{M^n}{ r^n}  \int_{B(y_0,r)} \frac{|x-x_0|}{|x-y|^{n+1} } \,dy\\ %\leq  C  \frac{ M^n}{ r^n  } \frac{r^n r}{(Mr)^{n+1}}\\
&\leq C\frac{1}{M r^n}.
\end{align*}
Hence $ |W_1(x)|\leq  C\frac{1}{ M r^n} \chi_{B(x_0,r)}(x)$.

We next estimate $W_2(x)$. From the definition of $g(x)$, we have
\begin{align*}
|W_2(x)| &= \chi_{B(y_0,r)}(x) |R_{N,l}h(x)|\\
&=\chi_{B(y_0,r)}(x)  \frac{1}{|R_{N,l}^*g(x_0)|} \left| \int_{B(x_0,r)} R_{N,l}(x,y) a(y) \,dy \right|\\
&=\chi_{B(y_0,r)}(x)  \frac{1}{|R_{N,l}^*g(x_0)|} \left| \int_{B(x_0,r)} R_{N,l}(x,y) a_+(y) \,dy \right|,
\end{align*}
where the last equality follows from the fact that $x\in B(y_0,r)\subset \mathbb{R}^n_+$ and from the definition of the Riesz kernel $R_N(x,y)$ as in \eqref{Riesz kernel}.
Hence, from the cancellation property of $a_+(y)$, we get
\begin{align*}
|W_2(x)| &=\chi_{B(y_0,r)}(x)  \frac{1}{|R_{N,l}^*g(x_0)|} \bigg| \int_{B(x_0,r)} (R_{N,l}(x,y)-R_{N,l}(x,x_0) )a_+(y) \,dy \bigg|\\
&\leq C\chi_{B(y_0,r)}(x) M^{n}  \int_{B(x_0,r)}\|a\|_{L^\infty(\mathbb{R}^n)} \frac{ |y-x_0|}{|x-x_0|^{n+1} } \,dy\\
%& \leq  C\chi_{B(y_0,r)}(x) M^{n}  \frac{r}{(Mr)^{n+1} }\\
&\leq \frac{C}{M r^n } \chi_{B(y_0,r)}(x).
\end{align*}

Combining the estimates of $W_1$ and $W_2$, we obtain that
\begin{align}\label{size}
 \Big|a(x)-\left(h(x)R_{N,l}^*g(x)- g(x)R_{N,l}h(x)\right)\Big|\leq  \frac{C}{M r^n}(\chi_{B(x_0,r)}(x)+\chi_{B(y_0,r)}(x)).
\end{align}
Next we point out that
\begin{align}\label{cancellation R}
&\int \Big[a(x)-\left(h(x)R_{N,l}^*g(x)- g(x)R_{N,l}h(x)\right) \Big] dx\\
&= \int a(x) dx - \int \left(h(x)R_{N,l}^*g(x)- g(x)R_{N,l}h(x)\right) dx\nonumber \\
&= 0,\nonumber
\end{align}
since $a(x)$ has cancellation (Proposition \ref{prop cancellation}) and the second integral equals 0 just by the definitions of $g$ and $h$.

Then the size estimate \eqref{size} and the cancellation \eqref{cancellation R},  together with Lemma \ref{lemma Hardy}, imply that
$$ \Big\|a(x)-\left(h(x)R_{N,l}^*g(x)- g(x)R_{N,l}h(x)\right)\Big\|_{H^1_{\Delta_N}(\mathbb{R}^n)} \leq C \frac{\log M}{M} <C\epsilon. $$
This proves the result for $1\leq l\leq n-1$.

We now consider the  the bilinear form $\Pi_n(g,h)$ and its approximation to $a(x)$.  Again, for the ball $B(x_0,r)$, we now consider the following cases:  Case 1: $x_{0,n}\geq 0$; \ \  Case 2: $x_{0,n}<0$. 

It suffices to consider
the Case 1 since the other can be handled symmetrically.  In this case, for $x_0$ with $x_{0,n}\geq 0$,
 choose  $y_0$ such that  $ y_{0,i}-x_{0,i}=\frac{Mr}{\sqrt{n}}$ for $i=1,\ldots,n$.

We now define the functions $g$ and $h$ as in \eqref{gh}. 
This, together with  Proposition \ref{p:RieszKernel}, yields  
\begin{align*}
R_{N,l}^*g(x_0) &=\left|\int_{B(y_0,r)}R_{N,n}(y,x_0)\,dy \right|
\\&=C_n \left|\int_{B(y_0,r)} \left( \frac{y_n-x_{0,n}}{|x_0-y|^{n+1}} +  \frac{x_{0,n}+y_n}{(|x'_0-y'|^2+|x_{0,n}+y_n|^2)^{\frac{n+1}{2}}} \right)\,dy\right| \\
&\geq C_n \left|\int_{B(y_0,r)}  \frac{y_n-x_{0,n}}{|x_0-y|^{n+1}}  \,dy\right| \\
&= C_n \left|y_n-x_{0,n}\right| \left|\int_{B(y_0,r)} \frac{1}{ |x_0-y|^{n+1}} \,dy\right| \\
&\geq C M^{-n}.
\end{align*}
Hene, we  obtain that
the claim \eqref{claim degenerate} holds for these $g$ and $h$.

Now  following the approximation as that for $R_{N,l}$ with $1\leq l\leq n-1$, we obtain that
\begin{align}\label{eeee}
\Big\|a(x)-\left(h(x)R_{N,l}^*g(x)- g(x)R_{N,l}h(x)\right)\Big\|_{H^1_{\Delta_N}(\mathbb{R}^n)} \leq C \frac{\log M}{M} <C\epsilon. 
\end{align}

\end{proof}

With this approximation result, we can now prove the main Theorem \ref{weakfactorization}, restated below for the convenience of the reader.

%\color{red}
%We need to modify $\Pi$ to be $\Pi_l$.
%\color{black}

\begin{theorem}
\label{t:UchiyamaFactor}
Suppose $1\leq l \leq n$.  For any $f\in H^1_{\Delta_{N}}(\mathbb{R}^n)$ there exists sequences $\{\lambda_j^k\}\in \ell^1$ and functions $g_j^{k}, h_j^{k}\in L^\infty(\mathbb{R}^n)$ with compact supports such that $$f=\sum_{k=1}^\infty\sum_{j=1}^{\infty} \lambda_{j}^{k}\Pi_l(g_j^{k}, h_j^{k}).$$  Moreover, we have that:
$$
\left\Vert f\right\Vert_{H^1_{\Delta_{N}}(\mathbb{R}^n)} \approx \inf\left\{\sum_{k=1}^\infty \sum_{j=1}^{\infty} \left\vert \lambda_j^k\right\vert \left\Vert g_j^k\right\Vert_{L^2(\mathbb{R}^n)}\left\Vert h_j^k\right\Vert_{L^2(\mathbb{R}^n)}: f=\sum_{k=1}^\infty\sum_{j=1}^{\infty}\lambda_j^k \,\Pi_l(g_j, h_j) \right\}.
$$
\end{theorem}
\begin{proof}
By Theorem \ref{t:upperbound} we have that $\left\Vert \Pi_l(g,h)\right\Vert_{H^1_{\Delta_{N}}(\mathbb{R}^n)}\leq C \left\Vert g\right\Vert_{L^2(\mathbb{R}^n)} \left\Vert h\right\Vert_{L^2(\mathbb{R}^n)}$, it is immediate that we have for any representation of $f=\sum_{k=1}^\infty\sum_{j=1}^\infty\lambda_j^k\, \Pi_l(g_j^{k}, h_j^{k})$ that
$$
\left\Vert f\right\Vert_{H^1_{\Delta_{N}}(\mathbb{R}^n)} \leq C \inf\left\{\sum_{k=1}^\infty\sum_{j=1}^{\infty}\left\vert \lambda_j^k\right\vert \left\Vert g_j^k\right\Vert_{L^2(\mathbb{R}^n)}\left\Vert h_j^k\right\Vert_{L^2(\mathbb{R}^n)}: f=\sum_{k=1}^\infty\sum_{j=1}^{\infty}\lambda_j^k \,\Pi_l(g_j^k, h_j^k) \right\}.
$$

We turn to show that the other inequality hold and that it is possible to obtain such a decomposition for any $f\in H^1_{\Delta_{N}}(\mathbb{R}^n)$.  By the atomic decomposition for $H^1_{\Delta_{N}}(\mathbb{R}^n)$, Theorem \ref{theorem characterization}, for any $f\in H^1_{\Delta_{N}}(\mathbb{R}^n)$ we can find a sequence $\{\lambda_{j}^{1}\}\in \ell^1$ and sequence of $H^1_{\Delta_{N}}(\mathbb{R}^n)$-atoms $a_j^{1}$ so that $f=\sum_{j=1}^{\infty} \lambda_j^{1} a_{j}^{1}$ and $\sum_{j=1}^{\infty} \left\vert \lambda_j^{1}\right\vert \leq C_0 \left\Vert f\right\Vert_{H^1_{\Delta_{N}}(\mathbb{R})}$.

We explicitly track the implied absolute constant $C_0$ appearing from the atomic decomposition since it will play a role in the convergence of the approach.  Fix $\varepsilon>0$ so that $\varepsilon C_0<1$. Then we also have a large positive number $M$ with ${\log M \over M}<\epsilon$.  We apply Theorem \ref{thm:ApproxFactorization} to each atom $a_{j}^{1}$.  So there exists $g_j^{1}, h_j^{1}\in L^\infty(\mathbb{R}^n)$ with compact supports and satisfying  $\left\Vert g_j^{1}\right\Vert_{L^2(\mathbb{R}^n)}\left\Vert h_j^{1}\right\Vert_{L^2(\mathbb{R}^n)}\leq CM^n $ and
\begin{equation*}
\left\Vert a_j^{1}-\Pi_l(g_j^{1}, h_j^{1})\right\Vert_{H^1_{\Delta_{N}}(\mathbb{R}^n)}<\varepsilon\quad \forall j.
\end{equation*}
Now note that we have
\begin{eqnarray*}
f & = & \sum_{j=1}^{\infty} \lambda_{j}^{1} a_j^{1}=   \sum_{j=1}^{\infty} \lambda_{j}^{1} \,\Pi_l(g_j^{1}, h_j^{1})+\sum_{j=1}^{\infty} \lambda_{j}^{1} \left(a_j^{1}-\Pi_l(g_j^{1}, h_j^{1})\right) :=  M_1+E_1.
\end{eqnarray*}
Observe that we have
\begin{eqnarray*}
\left\Vert E_1\right\Vert_{H^1_{\Delta_{N}}(\mathbb{R}^n)} & \leq & \sum_{j=1}^{\infty} \left\vert \lambda_j^{1}\right\vert \left\Vert a_j^{1}-\Pi_l(g_j^{1}, h_j^{1})\right\Vert_{H^1_{\Delta_{N}}(\mathbb{R}^n)} \leq  \varepsilon \sum_{j=1}^{\infty} \left\vert \lambda_j^{1}\right\vert \leq \varepsilon C_0\left\Vert f\right\Vert_{H^1_{\Delta_{N}}(\mathbb{R}^n)}.
\end{eqnarray*}
We now iterate the construction on the function $E_1$.  Since $E_1\in H^1_{\Delta_{N}}(\mathbb{R}^n)$, we can apply the atomic decomposition in $H^1_{\Delta_{N}}(\mathbb{R}^n)$, Theorem \ref{theorem characterization}, to find a sequence $\{\lambda_j^{2}\}\in \ell^1$ and a sequence of $H^1_{\Delta_{N}}(\mathbb{R}^n)$-atoms $\{a_j^{2}\}$ so that $E_1=\sum_{j=1}^{\infty} \lambda_j^{2} a_{j}^{2}$ and
$$
\sum_{j=1}^{\infty} \left\vert \lambda_j^{2}\right\vert \leq C_0 \left\Vert E_1\right\Vert_{H^1_{\Delta_{N}}(\mathbb{R}^n)}\leq \varepsilon C_0^2 \left\Vert f\right\Vert_{H^1_{\Delta_{N}}(\mathbb{R}^n)}.
$$

Again, we will apply Theorem \ref{thm:ApproxFactorization} to each atom $a_{j}^{2}$.  So there exist $g_j^{2}, h_j^{2}\in L^\infty(\mathbb{R}^n)$ with compact supports and satisfying  $\left\Vert g_j^{2}\right\Vert_{L^2(\mathbb{R}^n)}\left\Vert h_j^{2}\right\Vert_{L^2(\mathbb{R}^n)}\leq C M^n$ and
\begin{equation*}
\left\Vert a_j^{2}-\Pi_l(g_j^{2}, h_j^{2})\right\Vert_{H^1_{\Delta_{N}}(\mathbb{R}^n)}<\varepsilon,\quad \forall j.
\end{equation*}
We then have that:
\begin{eqnarray*}
E_1 & = & \sum_{j=1}^{\infty} \lambda_{j}^{2} a_j^{2}=   \sum_{j=1}^{\infty} \lambda_{j}^{2}\, \Pi_l(g_j^{2}, h_j^{2})+\sum_{j=1}^{\infty} \lambda_{j}^{2} \left(a_j^{2}-\Pi_l(g_j^{2}, h_j^{2})\right) :=  M_2+E_2.
\end{eqnarray*}
But, as before observe that
\begin{eqnarray*}
\left\Vert E_2\right\Vert_{H^1_{\Delta_{N}}(\mathbb{R}^n)} & \leq & \sum_{j=1}^{\infty} \left\vert \lambda_j^{2}\right\vert \left\Vert a_j^{2}-\Pi_l(g_j^{2}, h_j^{2})\right\Vert_{H^1_{\Delta_{N}}(\mathbb{R}^n)} \leq  \varepsilon \sum_{j=1}^{\infty} \left\vert \lambda_j^{2}\right\vert \leq \left(\varepsilon C_0\right)^{2}\left\Vert f\right\Vert_{H^1_{\Delta_{N}}(\mathbb{R}^n)}.
\end{eqnarray*}
And, this implies for $f$ that we have:
\begin{eqnarray*}
f & = & \sum_{j=1}^{\infty} \lambda_{j}^{1} a_j^{1} =   \sum_{j=1}^{\infty} \lambda_{j}^{1} \,\Pi_l(g_j^{1}, h_j^{1})+\sum_{j=1}^{\infty} \lambda_{j}^{1} \left(a_j^{1}-\Pi_l(g_j^{1}, h_j^{1})\right)\\
 & = & M_1+E_1=M_1+M_2+E_2 =  \sum_{k=1}^{2} \sum_{j=1}^{\infty} \lambda_{j}^{k} \,\Pi_l(g_j^{k}, h_j^{k})+E_2.
\end{eqnarray*}

Repeating this construction for each $1\leq k\leq K$ produces functions $g_j^{k}, h_j^{k}\in L^\infty(\mathbb{R}^n)$ with compact supports and  satisfying $\left\Vert g_j^{k}\right\Vert_{L^2(\mathbb{R}^n)}\left\Vert h_j^{k}\right\Vert_{L^2(\mathbb{R}^n)}\leq C M^n $ for all $j$, sequences $\{\lambda_{j}^{k}\}\in \ell^1$ with $\left\Vert \{\lambda_{j}^{k}\}\right\Vert_{\ell^1}\leq \varepsilon^{k-1} C_0^k \left\Vert f\right\Vert_{H^1_{\Delta_{N}}(\mathbb{R}^n)}$, and a function $E_K\in H^1_{\Delta_{N}}(\mathbb{R}^n)$ with $\left\Vert E_K\right\Vert_{H^1_{\Delta_{N}}(\mathbb{R}^n)}\leq \left(\varepsilon C_0\right)^{K}\left\Vert f\right\Vert_{H^1_{\Delta_{N}}(\mathbb{R}^n)}$ so that
$$
f=\sum_{k=1}^{K} \sum_{j=1}^\infty \lambda_{j}^{k} \,\Pi_l(g_j^{k}, h_j^{k})+E_K.
$$
Passing $K\to\infty$ gives the desired decomposition of $f=\sum_{k=1}^{\infty} \sum_{j=1}^\infty \lambda_{j}^{k}\, \Pi_l(g_j^{k}, h_j^{k})$.  We also have that:
$$
\sum_{k=1}^{\infty} \sum_{j=1}^\infty\left\vert \lambda_{j}^{k}\right\vert \leq \sum_{k=1}^{\infty} \varepsilon^{-1} (\varepsilon C_0)^{k} \left\Vert f\right\Vert_{H^1_{\Delta_{N}}(\mathbb{R}^n)}= \frac{ C_0}{1-\varepsilon C_0}\left\Vert f\right\Vert_{H^1_{\Delta_{N}}(\mathbb{R}^n)}.
$$
\end{proof}

Finally, we dispense with the proof of Theorem \ref{c:bmo}.

\begin{proof}[Proof of Theorem \ref{c:bmo}]
The upper bound in this theorem is contained in Theorem \ref{t:upperbound}.  

For the lower bound, we first note that from Theorem \ref{theorem characterization}, $H^1_{\Delta_N}(\mathbb{R}^n)$ 
has equivalent characterizations via atoms, which shows that $H^1_{\Delta_N}(\mathbb{R}^n)\cap L^\infty_{c}(\mathbb{R}^n)$
is dense in $H^1_{\Delta_N}(\mathbb{R}^n)$ with respect to the $H^1_{\Delta_N}(\mathbb{R}^n)$ 
norm, where we use $ L^\infty_{c}(\mathbb{R}^n)$ to denote the $L^\infty$ function with compact supports.

Then using the weak factorization in Theorem \ref{weakfactorization} we have that for $f\in H^1_{\Delta_N}(\mathbb{R}^n)\cap  L^\infty_{c}(\mathbb{R}^n)$, 
\begin{eqnarray*}
\left|\left\langle b,f\right\rangle_{L^2(\mathbb{R}^n)}\right|  \leq  \sum_{k=1}^\infty\sum_{j=1}^\infty \left|\lambda_j^k\right| \left|\left\langle b, \Pi_l(g_j^k, h_j^k)\right\rangle_{L^2(\mathbb{R}^n)}\right| =  \sum_{k=1}^\infty\sum_{j=1}^\infty \left|\lambda_j^k\right| \left|\left\langle  g_j^k, [b,R_{N,l}]h_j^k\right\rangle_{L^2(\mathbb{R}^n)}\right|.
\end{eqnarray*}
Hence we have that
\begin{eqnarray*}
\left\vert \left\langle b,f\right\rangle_{L^2(\mathbb{R}^n)}\right\vert &  \leq  & \sum_{k=1}^{\infty}\sum_{j=1}^\infty \left\vert \lambda_j^k\right\vert \left\Vert [b,R_{N,l}](h_j^k)\right\Vert_{L^2(\mathbb{R}^n)}\left\Vert g_j^k\right\Vert_{L^2(\mathbb{R}^n)}\\
& \leq & \left\Vert [b,R_{N,l}] :L^2(\mathbb{R}^n)\to L^2(\mathbb{R}^n)\right\Vert \sum_{k=1}^{\infty}\sum_{j=1}^\infty \left\vert \lambda_j^k\right\vert \left\Vert g_j^k\right\Vert_{L^2(\mathbb{R}^n)}\left\Vert h_j^k\right\Vert_{L^2(\mathbb{R}^n)}\\
& \leq  & C \left\Vert [b,R_{N,l}] :L^2(\mathbb{R}^n)\to L^2(\mathbb{R}^n)\right\Vert \left\Vert f\right\Vert_{H^1_{\Delta_N}(\mathbb{R}^n)}.
\end{eqnarray*}
By the duality between $\rm BMO_{\Delta_N}(\mathbb{R}^n)$ and $H^1_{\Delta_N}(\mathbb{R}^n)$ we have that:
$$
\left\Vert b\right\Vert_{\rm BMO_{\Delta_N}(\mathbb{R}^n)}\approx \sup_{\left\Vert f\right\Vert_{H^1_{\Delta_N}(\mathbb{R}^n)}\leq 1} \left\vert \left\langle b,f\right\rangle_{L^2(\mathbb{R}^n)}\right\vert\leq C \left\Vert [b,R_{N,l}] :L^2(\mathbb{R}^n)\to L^2(\mathbb{R}^n)\right\Vert.
$$
\end{proof}

\section{The fractional integrals: proof of Theorem \ref{c:bmo 1}}

Suppose $b$ is in $\rm BMO_{\Delta_N}(\mathbb{R}^n)$. Then according to \cite{DDSY}*{Proposition 4.2}, we have that
$b_{+,e}\in \rm BMO(\mathbb{R}^n)$ and $b_{-,e}\in \rm BMO(\mathbb{R}^n)$, and moreover,
$$ \|b\|_{\rm BMO_{\Delta_N}(\mathbb{R}^n)} \approx  \| b_{+,e}\|_{\rm BMO(\mathbb{R}^n)} + \|b_{-,e}\|_{\rm BMO(\mathbb{R}^n)}. $$

For every $f\in L^p(\mathbb{R}^n)$, we have
\begin{align*}
\|  [b,\Delta_N^{-\alpha/2}](f) \|_{L^q(\mathbb{R}^n)}^q = \int _{\mathbb{R}^n_+} [b,\Delta_N^{-\alpha/2}](f)(x)^q \,dx+ \int _{\mathbb{R}^n_-} [b,\Delta_N^{-\alpha/2}](f)(x)^q \,dx=: I+II.
\end{align*}
For the term $I$, note that when $x\in\mathbb{R}^n_+$, we have
\begin{align*}
[b,\Delta_N^{-\alpha/2}](f)(x)&= b(x)\Delta_N^{-\alpha/2}(f)(x) - \Delta_N^{-\alpha/2}(bf)(x)\\
&= b_{+,e}(x) \Delta^{-\alpha/2}(f_{+,e})(x) - \Delta^{-\alpha/2}(b_{+,e}f_{+,e})(x)\\
&= [b_{+,e},\Delta^{-\alpha/2}](f_{+,e})(x),
\end{align*}
which implies that
\begin{align*}
I &= \int _{\mathbb{R}^n_+} [b,\Delta^{-\alpha/2}](f)(x)^q \,dx =   \int _{\mathbb{R}^n_+} [b_{+,e},\Delta^{-\alpha/2}](f_{+,e})(x)^q \,dx \\
& \leq  \int _{\mathbb{R}^n} [b_{+,e},\Delta^{-\alpha/2}](f_{+,e})(x)^q \,dx\\
&\leq C   \| b_{+,e} \|_{\rm BMO(\mathbb{R}^n)}^q   \| f_{+,e} \|_{L^p(\mathbb{R}^n)}^q.
\end{align*}

For the last estimate we use the result \cite{CRW}*{Theorem 1}, which applies since we know from Proposition \ref{Riesz kernel} that $R_{N,l}$ is a Calder\'on--Zygmund kernel.  Similarly we can obtain that
\begin{align*}
II
&\leq C   \| b_{-,e} \|_{\rm BMO(\mathbb{R}^n)}^q   \| f_{-,e} \|_{L^2(\mathbb{R}^n)}^q.
\end{align*}
Combining the estimates for $I$ and $II$ above, we obtain that
\begin{align*}
\left\| [b,\Delta_N^{-\alpha/2}](f) \right\|_{L^q(\mathbb{R}^n)}^q &\leq C   \| b_{+,e} \|_{\rm BMO(\mathbb{R}^n)}^q   \| f_{+,e} \|_{L^p(\mathbb{R}^n)}^q+ C\| b_{-,e} \|_{\rm BMO(\mathbb{R}^n)}^q   \| f_{-,e} \|_{L^p(\mathbb{R}^n)}^q\\
&\leq  C\| b \|_{\rm BMO_{\Delta_N}(\mathbb{R}^n)}^q  \Big( \| f_{+,e} \|_{L^p(\mathbb{R}^n)}^q+ \| f_{-,e} \|_{L^p(\mathbb{R}^n)}^q\Big)\\
&\leq  C\| b \|_{\rm BMO_{\Delta_N}(\mathbb{R}^n)}^q  \| f \|_{L^p(\mathbb{R}^n)}^q,
\end{align*}
which yields that
$
\left\|[b,R_{N,l}] : L^p(\mathbb{R}^n)\to L^q(\mathbb{R}^n)\right\| \leq C\|b\|_{\rm BMO_{\Delta_N}(\mathbb{R}^n)}.$

\bigskip

{\bf Acknowledgments:} 
The authors would like to thank the referee for careful reading of this paper and for the helpful comments and suggestions, which made this paper more accurate.
 The authors would like to thank Professors Xuan Thinh Duong and Dongyong Yang for helpful discussions and for Professors Lixin Yan and Liang Song for providing their paper \cite{SY}, which hadn't appeared when the authors started this project in January 2015.

\end{document}